\documentclass{amsart}

\usepackage{amssymb,amsmath,amsthm}
\usepackage{mathrsfs}
\usepackage{thmtools}
\usepackage{array}
\usepackage{caption}
\usepackage{tikz-cd}
\usetikzlibrary{calc}
\usepackage[shortlabels]{enumitem}
\usepackage{todonotes}
\usepackage{subcaption}
\usepackage{upgreek}
\usepackage{tikz}
\usepackage{adjustbox}
\usepackage{longtable,booktabs}
\usepackage{tabu}

\definecolor{mblue}{RGB}{65,105,225}

\usepackage[colorlinks=true,linkcolor=blue,citecolor=mblue]{hyperref}

\usepackage[nameinlink,capitalise,noabbrev]{cleveref}

\crefname{equation}{}{}

\newtheorem{thm}{Theorem}[section]
\newtheorem*{thm*}{Theorem}
\newtheorem{cor}[thm]{Corollary}
\newtheorem*{cor*}{Corollary}
\newtheorem{prop}[thm]{Proposition}
\newtheorem{proposition}[thm]{Proposition}
\newtheorem*{proposition*}{Proposition}
\newtheorem{lem}[thm]{Lemma}
\newtheorem{lemma}[thm]{Lemma}

\newtheorem{maintheorem}{Theorem} 
 
\crefname{maintheorem}{Theorem}{Theorems}

\theoremstyle{definition}
\newtheorem{defn}[thm]{Definition}

\newtheorem{eg}[thm]{Example}
\newtheorem{example}[thm]{Example}
\newtheorem{notn}[thm]{Notation}

\newtheorem{rmk}[thm]{Remark}

\makeatletter
\let\c@equation\c@thm
\makeatother

\newenvironment{pf}{\begin{proof}}{\end{proof}}

\newcommand{\cA}{\ensuremath{\mathcal{A}}}
\newcommand{\C}{\ensuremath{\mathbb{C}}}

\newcommand{\cE}{\ensuremath{\mathcal{E}}}

\newcommand{\F}{\ensuremath{\mathbb{F}}}

\newcommand{\bH}{\ensuremath{\mathbf{H}}}
\newcommand{\rH}{\ensuremath{\mathrm{H}}}

\newcommand{\bP}{\ensuremath{\mathbb{P}}}

\newcommand{\R}{\ensuremath{\mathbb{R}}}
\newcommand{\bX}{\ensuremath{\mathbf{X}}}

\newcommand{\Z}{\ensuremath{\mathbb{Z}}}

\newcommand{\into}{\hookrightarrow}

\newcommand{\rtarr}{\longrightarrow}

\newcommand{\xrtarr}[1]{\xrightarrow{#1}}

\newcommand{\iso}{\cong}
\newcommand{\mr}[1]{\mathrm{#1}}

\newcommand{\cl}{\mathrm{cl}}
\newcommand{\Sp}{\mathbf{Sp}} 

\newcommand{\Ctwo}{{\mr{C}_2}}

\newcommand{\bMC}{\mathbf{H}_\bigstar}
\newcommand{\tMC}{\tilde{\mathbf{H}}_\bigstar}

\newcommand{\classicalcolor}{olive}
\newcommand{\equivcolor}{purple!50!blue}
\newcommand{\cxi}{{\color{\classicalcolor}{\boldsymbol\xi}}}

\newcommand{\exi}{{\color{\equivcolor}{\boldsymbol\xi}}}
\newcommand{\etau}{{\color{\equivcolor}{\boldsymbol{\tau}}}}

\newcommand{\Htwo}[3][\bH]{%
  \begin{tikzpicture}
    \begin{scope}
      \node[inner sep=0pt,outer sep=0pt] (a) {\phantom{$#1$}};
      \clip (a.south west) rectangle ($(a.north)-(0.5\pgflinewidth,0)$);
      \node[inner sep=0pt,outer sep=0pt,text=#2]  {$#1$};
    \end{scope}
    \clip (a.south east) rectangle ($(a.north)-(0.5\pgflinewidth,0)$);
      \node[inner sep=0pt,outer sep=0pt,text=#3]  {$#1$};
  \end{tikzpicture}
}

\DeclareMathOperator{\Sq}{Sq}
\newcommand{\Ptwo}{\mathrm{Sym}^2}
\newcommand{\Pow}{\mathrm{Sym}}

\DeclareMathOperator{\Ext}{Ext}

\newcommand{\cAcl}{\cA^\cl}

\newcommand{\cAC}{\cA^{\Ctwo}}
\newcommand{\cEC}{\cE^{\Ctwo}}

\newcommand{\cAhC}{\cA^{h\Ctwo}}
\newcommand{\cEhC}{\cE^{h\Ctwo}}

\newcommand{\OmRhoSRho}{\Omega^\rho_+ S^{\rho+1}}

\newcommand{\asig}{a_\sigma}
\newcommand{\usig}{u_\sigma}

\newcommand{\Ccolor}{orange}
\newcommand{\Scolor}{blue}
\newcommand{\rhoC}{{\textcolor{\Ccolor}{\rho_C}}}
\newcommand{\sigC}{{\textcolor{\Ccolor}{\sigma_C}}}
\newcommand{\rhoS}{{\textcolor{\Scolor}{\rho_\Sigma}}}
\newcommand{\sigS}{{\textcolor{\Scolor}{\sigma_\Sigma}}}

\newcommand{\Res}{\Phi^e}

\newcommand{\Fix}{\Phi^{C_2}}
\newcommand{\mFix}{\hat\Phi^{C_2}}

\newcolumntype{C}{>{$}c<{$}}
\newcolumntype{L}{>{$}l<{$}}
\newcolumntype{R}{>{$}r<{$}}
\newcolumntype{H}{>{\setbox0=\hbox\bgroup$}c<{$\egroup}@{}}

\numberwithin{equation}{section}
\numberwithin{figure}{section}
\numberwithin{table}{section}

\definecolor{blue'}{HTML}{0072B2}
\definecolor{green'}{HTML}{009E73}
\definecolor{red'}{HTML}{D55E00}

\begin{document}
\title
{
The real Brown-Peterson homology of $\Omega^\rho S^{\rho + 1}$
}
\author{Christian Carrick}
\address{Mathematisches Institut\\ Universität Bonn\\ 53129 Bonn, Germany}
\email{carrick@math.uni-bonn.de}
\author{Bertrand {J}. Guillou}
\address{Department of Mathematics\\ University of Kentucky\\
Lexington, KY 40506, USA}
\email{bertguillou@uky.edu}
\author{Sarah Petersen}
\address{Department of Mathematics\\ University of Colorado Boulder \\ 
Boulder, CO 80309}
\email{sarahllpetersen@gmail.com}

\begin{abstract}
We compute the $RO(C_2)$-graded real Brown--Peterson homology of the representation-loop space $\Omega^\rho S^{\rho + 1}$, where $\rho$ is the regular representation of the cyclic group of order two. This calculation gives a $C_2$-equivariant analogue of the classical computation of Brown--Peterson homology of the double loop space $\Omega^2 S^3$ due to Ravenel.
Along the way, we develop comodule Nishida relations for $\rho$-loop spaces.
\end{abstract}

\date{\today}

\maketitle

\setcounter{tocdepth}{1}
\tableofcontents

\section{Introduction}  

Classically, Dieudonn\'{e} theory gives an equivalence between the category of graded bicommutative Hopf algebras over $\mathbb{F}_p$ and the category of Dieudonn\'{e} modules over a certain ring $R$. In \cite{Goerss98}, Goerss showed that this equivalence respects the ring structure of Hopf rings over $\mathbb{F}_p$. It follows that Dieudonn\'{e} theory can be used to study the homology of spaces representing a multiplicative generalized cohomology theory. The homology of the representing spaces inherits the structure of a Hopf ring if the homology theory has a K\"{u}nneth isomorphism.

In the setting of equivariant homotopy theory, there is no immediate candidate for a Dieudonn\'{e} functor. 
Thus, in this paper, we study an equivariant computation which may be used to inform the development of an equivariant Dieudonn\'{e} theory. 
This approach is motivated by nonequivariant results.

Nonequivariantly, Goerss showed that Ravenel's computation \cite[Theorem~C]{Ravenel93}
 of the Brown--Peterson homology of $\Omega^2 S^3$
can be combined with Dieudonn\'{e} theory to describe the homology of the spaces representing a Landweber exact theory \cite[Sections 10 and 11]{Goerss98}. 
Ravenel's answer was 
\begin{equation}
\label{RavenelComputation}    
    BP_* \Omega^2 S^3 \iso BP_*[x_0,y_1,y_2,y_3,\dots]/(x_0^2,r_1,r_2,r_3,\dots),
\end{equation}
with $|x_0|=1$ and $|y_n|=2(2^n-1)$, and where the relations $r_n$ were determined modulo the square of $I=(v_0,v_1,v_2,\dots)\subset BP_*$ to be
\[
    r_n \equiv \sum_{i=0}^n v_i e_{n-i}^{2^i} \pmod{I^2}.
\]

By completing an analogous $C_2$-equivariant computation, 
we give a starting point for determining the {resulting} constraints 
{that may be imposed} on the existence of a Dieudonn\'{e} functor in the setting of equivariant homotopy theory.

In this paper, we compute the $RO(C_2)$-graded real Brown--Peterson homology $BP_\R$ (\cite{HK}) of the representation loop space $\Omega^\rho S^{\rho +1}$, where $\rho$ is the regular representation for the cyclic group of order two, $C_2$. As input for this calculation, we determine  
in \cref{prop:A-comod0rhoSrho+1} the
coaction of the equivariant dual Steenrod algebra $\cAC_\bigstar$ on the homology $\bH_\bigstar \Omega^\rho S^{\rho+1}$. The latter was computed by Behrens and Wilson \cite[Theorem~4.1]{BW}. The homology generators may all be described using Dyer-Lashof operations, and we produce comodule-Nishida relations, which prescribe the $\cAC_\bigstar$-coaction on the output of Dyer-Lashof operations.

This differs from equivariant co-Nishida relations already in the literature, which deal with equivariant {\it infinite} loop spaces rather than finite loop spaces \cite{Wilson17,Wilson19}. 
See also the thesis of Sikora \cite{Sikora21} for related work in the direction of finite loop spaces.

The context for these Nishida relations is $C_2$-spaces equipped with an action of an $E_\rho$-operad \cite{GuMa}, such as $\rho$-loop spaces. We construct Dyer-Lashof operations in the homology of $E_\rho$-spaces, which we denote
\[
        \bH_{k\rho}X \xrtarr{Q_0} \bH_{(2k)\rho}X,
    \qquad
    \qquad
    \bH_{k\rho}X \xrtarr{Q_1} \bH_{2k\rho+\sigma}X,
\]
\[
        \bH_{k\rho+1}X \xrtarr{Q_0} \bH_{(2k+1)\rho}X,
    \qquad \text{and}
    \qquad
    \bH_{k\rho+1}X \xrtarr{Q_1} \bH_{(2k+1)\rho+1}X.
\]
As a consequence of our co-Nishida relations (\cref{thm:equivarRightCoactionFormula}), we deduce (see \cref{eg:CoactQzero} and \cref{cor:CoactQdegreekrhoplusone}) the following precise formulas relating the comodule structure to Dyer-Lashof operations.

\begin{maintheorem}
    Let $X$ be an $E_\rho$-algebra. Then for  $x\in \rH_{k\rho}X$ and $y\in \rH_{k\rho+1}X$, the (right) coaction satisfies
\[
    \psi_R Q_0(x) = Q_0 \psi_R(x), 
    \qquad 
    \psi_R Q_1(x) = Q_1 \psi_R(x),
\]
\[
    \psi_R Q_0(y) = Q_0 \psi_R(y), 
    \qquad \text{and} \qquad
    \psi_R Q_1(y) = Q_1 \psi_R(y) + Q_0 \psi_R(y) \cdot (1 \otimes \etau_0),
\]
where $\etau_0\in \cAC_\bigstar$ is in degree 1. 
\end{maintheorem}

These formulas  allow us to compute the $\cAC_\bigstar$-comodule structure on $\bMC \Omega^\rho S^{\rho + 1}$ (\cref{prop:A-comod0rhoSrho+1}), which serves as input to compute the equivariant Adams spectral sequence converging to ${BP_\R}_\bigstar \Omega^\rho S^{\rho+1}$.
 
 Observing that ${BP_\R}\wedge \OmRhoSRho$ is Borel complete (\cref{prop:borelcomplete}), we instead compute the Borel equivariant Adams spectral sequence developed by Greenlees \cite{G}. This is preferable because it leads to a cleaner computation. In order to compute the Adams $E_2$-page of the Borel equivariant Adams spectral sequence, we first compute the $\asig$-Bockstein spectral sequence. Our main computational result gives a $C_2$-equivariant analogue of Ravenel's computation of $BP_* \Omega^2 S^3$ and is described over the course of the following two corollaries as well as in charts in \cref{sec:charts}. 

\begin{maintheorem}[\cref{cor:EinftyBSS}]
    The $E_\infty$-page of the $\asig$-BSS for $BP_\R\wedge \OmRhoSRho$ is isomorphic to the subalgebra of 
\[
\frac{\F_2[\usig^\pm,\asig,v_0,v_1,\ldots,t_0,e_1,e_2,\ldots]}{(t_0^2+\usig e_1,r_1,r_2,\ldots,\asig^{2^{j+1}-1}v_j|j\ge0)}
\]
generated by
\[
\{\usig^{2^{j+1}k}v_j,\asig,t_0,e_1,e_2,\ldots|j\ge0,k\in\Z\}.
\]  
\end{maintheorem}

This gives the associated graded, with respect to the $\asig$-filtration, of 
\[\Ext_{\mathcal{E}^h_\bigstar}(\bH_\bigstar^h,\bH_\bigstar^h\Omega^\rho S^{\rho+1}),
\]
which is the $E_2$-term for the Borel Adams spectral sequence for $BP_\R\wedge\OmRhoSRho$. We also show that there are no nonzero Adams differentials.

\begin{maintheorem}[\cref{cor:adamscollapse}]
     The Borel Adams SS for $BP_\R\wedge \OmRhoSRho$ collapses on the $E_2$-page.
\end{maintheorem}

\begin{rmk}
\label{rmk:hiddenextns}

We consider some hidden multiplications in the $\asig$-BSS in \cref{subsec:hiddenextensions}. 
    One could attempt to resolve all extensions in the $\asig$-BSS and describe the $E_2$-page of the Borel Adams SS for $BP_\R\wedge\OmRhoSRho$ explicitly as a module over that of $BP_\R$. However, to calculate the ${BP_\R}_\bigstar$-module ${BP_\R}_\bigstar\Omega^\rho S^{\rho+1}$, one would also need to resolve all extensions in the Borel Adams SS for $BP_\R\wedge\OmRhoSRho$. Since these extensions are not even known nonequivariantly (see \cite[Conjecture 3.4]{Ravenel93}), we do not attempt to resolve them here. However, we remark that \cref{prop:resinjective} suggests that the equivariant extensions may be equivalent to the nonequivariant ones.
\end{rmk}

These computations give us a number of immediate qualitative consequences for $BP_\R\wedge \OmRhoSRho$. Nonequivariantly, the Snaith splitting for double loop spaces gives a stable decomposition of $\Sigma^\infty_+\Omega^2S^3$ into a wedge
\[\Sigma^\infty_+\Omega^2S^3\simeq \bigvee\limits_{r\ge0}\Sigma^\infty\mathrm{Sym}^r_2S^1,\]
where $\mathrm{Sym}^r_2S^1=C_2(r)_+\wedge_{\Sigma_r}S^r$ and $C_2(r)$ is the configuration space of $r$ points in $\R^2$. We may therefore decompose $\Sigma^\infty_+\Omega^2S^3$ into an even and odd part
\begin{align*}
    L_{ev}&:=\bigvee\limits_{r\ge0}\Sigma^\infty\mathrm{Sym}^{2r}_2S^1\\
    L_{odd}&:=\bigvee\limits_{r\ge0}\Sigma^\infty\mathrm{Sym}^{2r+1}_2S^1
\end{align*}
and an immediate consequence of Ravenel's computation of $BP_*\Omega^2S^3$ is that \newline $BP\wedge L_{ev}$ has homotopy groups concentrated in even degrees and 
\[BP\wedge L_{odd}\simeq \Sigma BP\wedge L_{ev}.\]

We lift this result to the $C_2$-equivariant setting using the equivariant Snaith splitting
(see \cref{thm:Snaith})
\[\Sigma^\infty_{C_2}\OmRhoSRho\simeq \bigvee\limits_{r\ge0}\Sigma^\infty\mathrm{Sym}^r_
\rho S^1,\]
where $\mathrm{Sym}^r_\rho S^1=C_\rho(r)_+\wedge_{\Sigma_r}S^r$ and $C_\rho(r)$ is the configuration space of $r$ points in the representation $\rho$. Grouping the even and odd components together as $L_{ev}$ and $L_{odd}$ as before, we have the following.

\begin{maintheorem}[\cref{prop:evenoddsplit,prop:resinjective}]
The $C_2$-spectrum $BP_\R\wedge L_{ev}$ is strongly even and
\[BP_\R\wedge \OmRhoSRho\simeq BP_\R\wedge (L_{ev}\vee \Sigma L_{ev}) .\]
In particular, the restriction map
\[\mathrm{res}:(BP_\R)_{*\rho}\Omega^\rho S^{\rho+1}\to BP_{2*}\Omega^2S^3\]
is an isomorphism.
\end{maintheorem}

In addition to the real Brown--Peterson homology studied in this paper, there is also a notion of $G$-equivariant Brown--Peterson homology for any compact lie group $G$. In the case where $G = C_2$, the real Brown--Peterson spectrum $BP_\mathbb{R}$ differs from the spectrum $BP_{C_2}$, though they are both (genuine) $C_2$-spectra. 
One important difference is that $BP_G$ is complex-oriented, while $BP_\R$ is Real-oriented (see \cite{HK}). For our computation, the most important point is that the cohomology of $BP_\R$ can be identified as a module over the Steenrod algebra as 
\[
    \bH^\bigstar BP_\R \iso \cAC\!/\!/ \cEC,
\]
so that a change-of-rings formula gives a simplified description of the $E_2$-term for the Adams spectral sequence (see \cref{sec:E2Term}).
We are not aware of a similar description of $\bH^\bigstar BP_{C_2}$.
Another variant of our computation would be to consider alternative equivariant lifts of $\Omega^2 S^3$. The signed loop space $\Omega^{2\sigma}S^{1+2\sigma}$ is another interesting lift. Surprisingly, these two equivariant lifts agree! See \cite[Remark~1.3]{HW} or \cite[Lemma~3.1]{Kl}.

\begin{rmk}
    We have used the Adams spectral sequence to compute the $BP_\R$-homology of $\Omega^\rho S^{\rho+1}$, with $E_2$-page given by Ext groups over $\cEC$. Typically, a good way to understand such calculations is to first compute Ext over the subalgebras $\cEC(n)$ for small examples of $n$. In the cases of $n=0$ and $1$, these would be $E_2$-terms for Adams spectral sequences computing the $H\Z$ or $k\R$-homology of the loop space, respectively. However, it turns out that these computations are in some sense more complicated than that for Ext over $\cEC$. See also \cref{rmk:E(n)nonregular}.
\end{rmk}

\subsection{Acknowledgments}
The authors would like to thank Andrew Baker, Mark Behrens, Prasit Bhattacharya, Mike Hill, Tyler Lawson, Douglas Ravenel, and Alex Waugh for enlightening conversations. The first and third author would like to thank the Isaac Newton Institute for Mathematical Sciences, Cambridge, for support and hospitality during the programme Equivariant homotopy theory in context, where work on this paper was undertaken. This work was supported by EPSRC grant EP/Z000580/1. 
The first author was supported by the National Science Foundation under Grant No. DMS-2401918.
The second author was supported by the National Science Foundation under Grants No. DMS-2003204 and DMS-2403798 as well as the Simons Foundation under Grant No. MPS-TSM-00007067.
The third author was supported by the National Science Foundation under Grant No. DMS-2135884.

\section{Notation}

\begin{itemize}
    \item We write $\bH$ for the $C_2$-equivariant Eilenberg-Mac~Lane spectrum $H_{C_2} \underline{\F}_2$ with constant Mackey functor coefficients
    \item We write $\bH^h$ for the Borel completion of $\bH$, i.e. $\bH^h=F({EC_2}_+,\bH)$.
    \item We write $\bMC$ and $\bH^h_\bigstar$ for the $RO(C_2)$-graded coefficients of these theories. See \cref{sec:C2Steenrod}.
    \item We write $H$ for the (nonequivariant) Eilenberg-Mac~Lane spectrum $H\F_2$.
    \item We write $\Res$ for the restriction, or underyling, functor $\Res\colon\Sp_{C_2} \to \Sp$.
    \item  We also write 
$
    \Res\colon \bH^{n+k\sigma}(X)
    \rtarr \rH^{n+k}(\Res X).
$
    \item We will denote by $\sigma$ the sign representation of $C_2$ and by $\rho$ the regular representation. This splits as $\rho \iso 1 \oplus \sigma$, where $1$ denotes the 1-dimensional trivial representation.
    \item $RO(C_2)$ is the real representation ring of $C_2$. This is a free abelian group on generators the trivial representation $1$ and the sign representation $\sigma$. 
    \item At times, we will need to consider sign and regular representations for two different groups of order two, which we will denote by $C_2$ and $\Sigma_2$. In these situations, we will write $\sigC$ and $\rhoC$ for the representations of $C_2$ and $\sigS$ and $\rhoS$ for the $\Sigma_2$-representations.
    \item  We write $\asig\colon S^{-\sigma}\to S^0$ for the desuspension of the inclusion of fixed points $S^0 \into S^\sigma$. We will use the same notation for its Hurewicz image in $\bH_{\sigma}$ and $\bH^h_{-\sigma}$. 
    \item  The orientation class in $\bH_1 S^\sigma$ corresponds, under the suspension isomorphism, to a class $\usig \in \bH_{1-\sigma}$. We use the same notation for its image in $\bH^h_{1-\sigma}$.
    \item We write $\cxi_i$ for the generators of the (nonequivariant) dual Steenrod algebra (see \cref{sec:dualSteenrod}) and $\exi_i$, $\etau_n$ for the generators of the $C_2$-equivariant dual Steenrod algebra (see \cref{sec:C2Steenrod})
    \item As noted in the introduction, $BP_\R$ is the Real Brown-Peterson spectrum introduced in \cite{HK}. By \cite{QZ}, this is an $E_\rho$-ring spectrum.
    \item We will also make use of the Real complex cobordism spectrum $MU_\R$. This is a $C_2$-$E_\infty$-ring spectrum by \cite{HHR}.
\end{itemize}

\section{Background}

{
We begin by reviewing the non-equivariant versions of tools to be used throughout this article.
}

\subsection{The dual Steenrod Algebra}
\label{sec:dualSteenrod}

Recall that the dual Steenrod algebra at the prime $p = 2$ is 
\[
    \cA_* \cong \F_2 [\cxi_1, \cxi_2, \cdots ]
    \]
    where $\vert \cxi_i \vert = 2^i - 1$
\cite{Milnor58}.

The generators $\cxi_i$ are defined using the (completed) right coaction of $\cA_*$ on the cohomology of the classifying space $B \Sigma_2$
\[
\hat \psi : \rH^* (B {\Sigma_2}) \to \rH^* (B {\Sigma_2} ) \widehat \otimes_{H_*} \cA_*
\]
according to the formula
\[
    \hat \psi(t) = \sum_{i\geq 0} t^{2^i} \otimes \cxi_i,
\]
where $t\in \rH^1(B\Sigma_2)$ is the generator.

\subsection{Dyer-Lashof operations}

An $E_n$-algebra, such as an $n$-fold loop space, inherits an action of Dyer-Lashof operations on its homology. 
See \cite[Section 5]{Law} for a clear account.
These are operations
\[
    Q^{k+j}\colon \rH_k X \rtarr \rH_{2k+j} X
\]
for $0 \leq j \leq n-1$. It is sometimes convenient to use the alternative notation
\[
    Q_{j}\colon \rH_k X \rtarr \rH_{2k+j} X.
\]
The operations $Q_j$ are additive for $j \leq n-2$.
However, in general, the top operation $Q_{n-1}$ fails to be additive, and this failure of additivity is measured by an operation known as the Browder bracket.
In lower index notation, the {\bf Cartan formula} can neatly be expressed as
\[
    Q_j(x\cdot y) = \sum_r Q_r(x) \cdot Q_{j-r}(y)
\]
when $j \leq n-2$.

The homology groups also have a natural left action of the Steenrod algebra $\cA$:
\[
    \Sq_i \colon \rH_k X \rtarr \rH_{k-i} X,
\]
which can be converted into a right action by use of the anti-automorphism $\chi\colon \cA \to \cA$. The right action of $\cA$ is related to the left action of the Dyer-Lashof operations according to the {\bf Nishida relations}: 
\[
    Q_j(x) \cdot \Sq_r = \sum_\ell \binom{\deg{x} +j-r}{r-2\ell} Q_{j-r+\ell}\left( x\cdot \Sq_r \right),
\]
again for $j \leq n-2$.

\begin{rmk}
    There is also a Cartan formula and Nishida relation for the top operation $Q_{n-1}$, though it involves the Browder bracket. If the Browder bracket vanishes on $\rH_* X$, then the Cartan formula and Nishida relation for $Q_{n-1}$ take the same form as those for the lower $Q_j$'s.
\end{rmk}

Alternatively, the action of the Dyer-Lashof operations can be related to the coaction of the dual Steenrod algebra $\cA_*$ via the {\bf co-Nishida relations}. In order to express these, recall that the $E_n$-extended power 
\[
    \Ptwo_n(S^k) = \mathcal{C}_n(2)_+ \wedge_{\Sigma_2} S^{k \rho_{\Sigma_2}},
\]
where $\mathcal{C}_n$ is an $E_n$-operad,
is equivalent to the suspended stunted projective space $\Sigma^k \R\bP^{k+n-1}_k$. Write 
\[
    e_{2k+j} \in \tilde\rH_{2k+j} \Ptwo_n(S^k), \quad \text{where }j \in \{0,\dots,n-1\},
\]
for the homology generators. 
Suppose that the right  $\cA_*$-coaction on $\tilde\rH_* \Ptwo_n S^k$ is given by $\psi_R(e_{2k+j}) = \sum_i e_{2k+i} \otimes a_i$. Then the co-Nishida relations \cite[(2.6)]{Baker15} are
\begin{equation}
\label{eq:classicalcoNishida}    
    \psi_R Q_j x = \sum_i (1 \otimes a_i) \cdot Q_i(\psi_R x).
\end{equation}
Note that evaluation of $Q_i(\psi_R x)$ involves the Cartan formula.

We will primarily work with the coaction not by $\cA_*$ but rather by the quotient $\cE_* \iso E(\xi_1,\xi_2,\dots)$. This is the quotient of $\cA_*$ by the ideal generated by the $\xi_i^2$.
We have an analogous co-Nishida relation for the $\cE_*$-coation:
\begin{equation}
\label{eq:classicalEcoNishida}
    \psi_R^{\cE_*} Q_j x = q\left(\sum_i (1 \otimes a_i) \cdot Q_i(\psi_R^{\cE_*} x)\right),
\end{equation}
where $q\colon \cA_* \to \cE_*$ is the quotient.
This follows because the Dyer-Lashof operations $Q_i$ descend along the quotient map $q$, since they preserve the ideal $(\xi_1^2,\xi_2^2,\dots)$. 
The latter follows from the formula
\[
    Q_j \xi_k^2 = 
    \begin{cases}
        (Q_{j/2} \xi_k)^2 & j \text{ even} \\
        0 & j \text{ odd,}
    \end{cases}
\]
which is a consequence of the Cartan formula.

\subsection{{The $\cA_*$-comodule ${\rH_* \Omega^2 S^3}$}}
\label{AcomodReview}

The homology of $\Omega^2 S^3$ is 
\[
	\rH_*(\Omega^2 S^3) \iso \F[x_1,x_2,x_3,\dots], 
\]
with $x_i$ in degree $2^i-1$. 
The space $\Omega^2 S^3$ is visibly a double loop space, but in fact it is a triple loop space, as $S^3$ is a topological group.
Thus its homology comes equipped with an action by Dyer-Lashof operations $Q_0$, $Q_1$, and $Q_2$. We will not make use of the top operation, $Q_2$.
Araki and Kudo showed \cite{AK,DL} that the generators are given by (what are now called) Dyer-Lashof operations on previous generators. For example, 
\[
	x_2 = Q^2 x_1 = Q_1 x_1, \qquad x_3 = Q^4 x_2 = Q_1 x_2, \qquad \text{etc.} 
\]
The (co-)Nishida relations can then be used to deduce the $\cA$-module, or $\cA_*$-comodule, structure on $\rH_*(\Omega^2 S^3)$. 

According to \eqref{eq:classicalEcoNishida}, we can inductively work out the $\cA_*$-coaction on the $x_i$'s.
For this, the coaction on the $E_2$-extended power $\Ptwo_2(S^k) \iso \Sigma^k \R\bP_k^{k+1}$ is needed. 
The homology class $e_{2k}$ is primitive, and the coaction on $e_{2k+1}$
 is given by
\[
    \psi_R(e_{2k+1}) = 
    \begin{cases}
        e_{2k+1} \otimes 1 + e_{2k} \otimes \xi_1 & k \text{ odd} \\
        e_{2k+1} \otimes 1 & k \text{ even.}
    \end{cases}
\]

\begin{prop}
    The right $\cA_*$-coaction on $\rH_*(\Omega^2 S^3)$ is given by 
    \begin{equation}
    \psi_R x_i = \sum_{k=0}^{i-1} x_{i-k}^{2^k} \otimes \xi_k.
\end{equation}
\end{prop}

\begin{proof}
We induct on $i.$ The base case for holds for degree reasons. Suppose the above formula holds for $x_{i}.$ Then the coNishida relations \eqref{eq:classicalEcoNishida} imply
\begin{align*}
    \psi_R(x_{i + 1}) &  = Q_1(\psi_R(x_i)) + (1 \otimes \xi_1 ) Q_0 (\psi_R x_i) \\
    & = Q_1 \left(\sum_{k=0}^{i-1} x_{i-k}^{2^k} \otimes \xi_k  \right) + ( 1 \otimes \xi_1) Q_0 \left( \sum_{k=0}^{i-1} x_{i-k}^{2^k} \otimes \xi_k \right ) \\
    & = Q_1 (x_i) \otimes \xi_0 + \sum_{k=0}^{i-1} Q_0( x_{i-k}^{2^k}) \otimes Q_1(\xi_k) +  \sum_{k=0}^{i-1} Q_0(x_{i-k}^{2^k}) \otimes \xi_1 Q_0(\xi_k)  \\
    & = x_{i + 1} \otimes 1 + \sum_{k=0}^{i-1} \left( x_{i-k}^{2^{k + 1}} \otimes (\xi_{k + 1} + \xi_1 \xi_k^2) +  x_{i-k}^{2^{k + 1}} \otimes \xi_1 \xi_k^2 \right) \\
    & = \sum_{k=0}^i x_{i+1-k}^{2^k} \otimes \xi_k,
\end{align*}
where in the third line we apply the Cartan formula and in the fourth line we use that $Q_1(\xi_k) = \xi_{k + 1} + \xi_1 \xi_k^2$ (see, for instance \cite[Lemma 4.4]{Baker15}).
\end{proof}

\section{The $C_2$-equivariant Steenrod algebra}
\label{sec:C2Steenrod}

\begin{figure}
    \centering
    \includegraphics{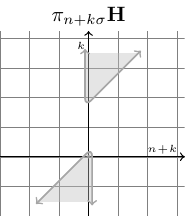}
    \qquad \qquad
    \includegraphics{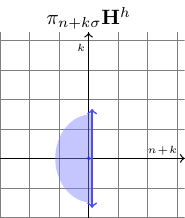}
    \caption{
    The $RO(C_2)$-graded coefficients $\bMC$ and $\bMC^h$, as described in \eqref{eq:Hcoeffs}. The class $\asig$ is in degrees $n+k=-1$ and $k=-1$ while $\usig$ is in degrees $n+k=0$ and $k=-1$.
    } 
    \label{fig:M2}
\end{figure}

We now review the $C_2$-equivariant Steenrod algebra as well as the Borel-equivariant form. These are algebras over $\bMC$ and $\bMC^h$, respectively.
We have isomorphisms
\begin{equation}
\label{eq:Hcoeffs}    
        \bMC \iso \F_2[\asig,\usig] \oplus \frac{\F_2[\asig,\usig]}{(\asig^\infty,\usig^\infty)}
        \qquad \text{and} \qquad
        \bMC^h \iso \F_2[\asig,\usig^{\pm 1}],
\end{equation}
and these are displayed in \cref{fig:M2}. The summand $\frac{\F_2[\asig,\usig]}{(\asig^\infty,\usig^\infty)}$ of $\bMC$ is often referred to as the ``negative cone''.

Recall that the equivariant dual Steenrod algebra, displayed in \cref{fig:AC}, is
\[
    \cAC_\bigstar \iso \bH_\bigstar[\etau_0,\etau_1,\dots,\exi_1,\exi_2,\dots]/
    (\etau_i^2 + \usig \exi_{i+1} + \asig \etau_0 \exi_{i+1} + \asig \etau_{i+1})
\]
where $\vert \etau_i \vert = 2^i \rho - \sigma$ and $\vert \exi_i \vert = (2^i - 1 )\rho$ \cite{HK}. 
We will also make use of the Borel equivariant dual Steenrod algebra
\[
    \cAhC_\bigstar \iso \bH_\bigstar^h[\etau_0,\etau_1,\dots,\exi_1,\exi_2,\dots]^{\,\widehat{}}_{\asig}/
    (\etau_i^2 + \usig \exi_{i+1} + \asig \etau_0 \exi_{i+1} + \asig \etau_{i+1}),
\]
which is displayed in \cref{fig:AhC}.
While $(\bH_\bigstar,\cAC_\bigstar)$ is a Hopf algebroid in $RO(C_2)$-graded $\F_2$-vector spaces, for flatness one must regard $(\bH_\bigstar^h,\cAhC_\bigstar)$ as a Hopf algebroid in the category $\mathcal{M}_{\asig}$ of $\asig$-complete $RO(C_2)$-graded $\F_2$-vector spaces. The map $(\bH_\bigstar,\cAC_\bigstar)\to (\bH^h_\bigstar,\cAhC_\bigstar)$ is obtained by inverting $\usig$ and completing at $\asig$, as shown by Hu--Kriz \cite[Corollary 6.40 and Theorem 6.41]{HK}, see also \cite[Theorem 2.14]{LSWX}. We refer the reader to \cite[Section 2]{LSWX} for more details on these two variants of the $C_2$-equivariant dual Steenrod algebra.

We recall now the definition of the generators $\etau_i$ and $\exi_i$ via the (completed) right coaction of $\cAC_\bigstar$ on the cohomology of the equivariant classifying space 
\[
    B_{C_2}\Sigma_2 \simeq S\big((\rhoC\otimes \sigS)^\infty\big)/\Sigma_2,
\]
where $\rhoC$ is the regular representation of $C_2$ and $\sigS$ is the sign representation of $\Sigma_2$.

\begin{figure}
    \centering
    \includegraphics[width=0.875\linewidth]{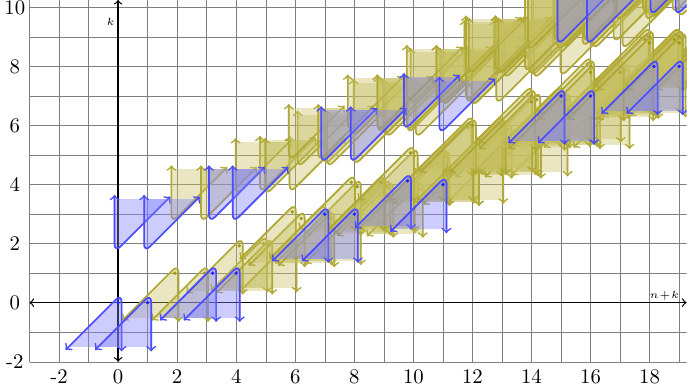}
    \caption{The $C_2$-equivariant dual Steenrod algebra $\cAC_{n+k\sigma}$, using the ``motivic'' grading in which the vertical direction indicates multiples of $\sigma$ and the horizontal is the underlying topological dimension. Here $\cEC_\bigstar$ is indicated in blue. Each copy of $\bMC$ contributes a ``positive'' cone (pointing down) and a ``negative'' cone (point up).} 
    \label{fig:AC}
\end{figure}

\begin{figure}
    \centering
    \includegraphics[width=0.875\linewidth]{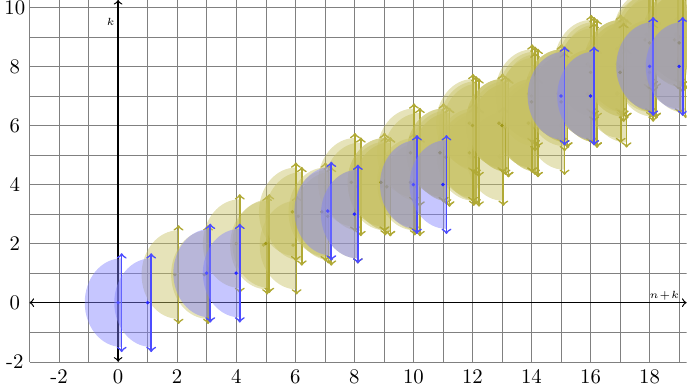}
    \caption{The Borel $C_2$-equivariant dual Steenrod algebra $\cAhC_{n+k\sigma}$, using the ``motivic'' grading in which the vertical direction indicates multiples of $\sigma$ and the horizontal is the underlying topological dimension. Here $\cEhC_\bigstar$ is indicated in blue. 
    Each copy of $\bH_\bigstar^h$ contributes a left half-plane, the $\usig$-localization of the positive cone in $\bMC$.}
    \label{fig:AhC}
\end{figure}

\begin{prop}[{\cite[Lemma~6.27]{HK},\cite[Theorem~6.10]{Voev}}] \label{prop:hom classifying space}
    The cohomology of $B_{C_2}\Sigma_2$ is
    \[
        \bH^\bigstar(B_{C_2}{\Sigma_2})\iso \bMC[c,d]/(c^2=\asig c+\usig d),
    \]
    where $\deg c=\sigma$ and $\deg d = 1+\sigma$. The class $d$ is the restriction of the Euler class of the tautological complex line bundle on $B_{C_2}S^\sigma$, 
    and the Bockstein applied to $c$ is $d$.
\end{prop}

\begin{rmk}
    \label{cAmbiguous}
    As explained in \cite[Proof of Theorem~2.12]{LSWX}, the class $c$ is not well-defined. There are two possible choices for $c$, whose difference is the element $a$. 
    However, the choice has no impact on the relation $c^2=\asig c+\usig d$ or the coaction on $c$ given below.
\end{rmk}

Then the completed coaction 
\[
    \hat\psi \colon \bH^\bigstar(B_{C_2}{\Sigma_2}) \rtarr \bH^\bigstar(B_{C_2}{\Sigma_2}) \widehat\otimes_{\bMC} \cAC_\bigstar
\]
 defines the elements $\etau_n$ and $\exi_n$ via the formulas
\begin{equation}
    \label{c-coaction}
    \hat\psi(c) = c\otimes 1 + \sum_{n \geq 0} d^{2^n} \otimes \etau_n
\end{equation}    
and
\begin{equation}
    \label{d-coaction}
    \hat\psi(d) = \sum_{n \geq 0} d^{2^n} \otimes \exi_n.
\end{equation}

\begin{lemma}
\label{lemma:xicomesfromMUR}
    For all $n\ge 1$, the class $\exi_n\in \cAC_\bigstar$ is in the image of the map
    \[{MU_\R}_\bigstar MU_\R\to \bH_\bigstar \bH= \cAC_\bigstar\]
    induced by the Postnikov truncation $MU_\R\to \bH$.
\end{lemma}

\begin{proof}
    The classes $\exi_n$ are defined in \cref{d-coaction} via the completed coaction on the class $d\in \bH^\rho(B_{C_2}{\Sigma_2})$, which is by definition restricted from the Euler class $d\in \bH^\rho(B_{C_2}S^\sigma)$.
    Writing $\C\bP^n_\R$ for $\C\bP^n$ equipped with the complex conjugation action of $C_2$,
    the completed coaction map 
    \[
    \hat\psi \colon \bH^\bigstar(B_{C_2}S^\sigma) \rtarr \bH^\bigstar(B_{C_2}S^\sigma) \widehat\otimes_{\bMC} \cAC_\bigstar
\]
is obtained as a limit of the right coactions
    \[
    \psi \colon \bH_\bigstar(D(\C\mathbb{P}^n_\R)) \rtarr \bH_\bigstar(D(\C\mathbb{P}^n_\R)) \otimes_{\bMC} \cAC_\bigstar
\]
under the Spanier-Whitehead duality identifications $\bH_\bigstar({D}(\C\mathbb{P}^n_\R))\cong \bH^\bigstar(\C\mathbb{P}^n_\R)$ and the identification $\C\mathbb{P}^\infty_\R\simeq B_{C_2}S^\sigma$. The latter coactions $\psi$ are obtained by applying homotopy groups and a Kunneth isomorphism to the bottom row of the commutative diagram
\[
\begin{tikzcd}
    D(\C\mathbb{P}^n_\R)\wedge MU_\R\arrow[d]\arrow[r,"1\wedge\eta\wedge1"]&D(\C\mathbb{P}^n_\R)\wedge MU_\R\wedge MU_\R\arrow[d]\\
   D(\C\mathbb{P}^n_\R)\wedge \bH\arrow[r,"1\wedge\eta\wedge1"]&D(\C\mathbb{P}^n_\R)\wedge \bH\wedge \bH
\end{tikzcd}
\]
induced by the Postnikov truncation $MU_\R\to \bH$. By naturality of the Kunneth isomorphisms, this gives a commutative diagram of right coactions
\[
\begin{tikzcd}
{MU_\R}_\bigstar D(\C\mathbb{P}^n_\R)\arrow[d]\arrow[r,"\psi"]&{MU_\R}_\bigstar D(\C\mathbb{P}^n_\R)\otimes_{{MU_\R}_\bigstar}{MU_\R}_\bigstar MU_\R\arrow[d]\\
    \bH_\bigstar D(\C\mathbb{P}^n_\R) \arrow[r,"\psi"]&\bH_\bigstar D(\C\mathbb{P}^n_\R)\otimes_{\bH_\bigstar}\cAC_\bigstar
\end{tikzcd}
\]
Since the Euler class $d$ lifts (to the universal orientation)
along the lefthand vertical map, 
so do the terms in its coaction by commutativity of the diagram, and the result follows.
\end{proof}

We will use the above coaction formulas for $\bH^\bigstar(B_{C_2}{\Sigma_2})$, together with knowledge of the underlying and fixed point homomorphisms on $\bH^\bigstar(B_{C_2}{\Sigma_2})$, to determine the underlying and fixed points homomorphisms on $\cAC_\bigstar$.
Recall from \cite[Section 2]{BW} that for a $C_2$-spectrum $\bX$ there are underlying and (modified) geometric fixed point homomorphisms
\[  
    \Res\colon \bMC \bX \rtarr H_* (\Res \bX)[u^{\pm 1}]
\]
and
\begin{equation}
\label{defn:mFix}
    \mFix\colon \bMC \bX \rtarr H_*(\Fix \bX)[a^{\pm 1}]
\end{equation}
and similarly in cohomology. The map $\mFix$ is called a modified geometric fixed point homomorphism as the the honest geometric fixed point homomorphism would have target $H_*(\mFix \bX)[\usig,\asig^{\pm 1}]$. The modified homomorphism is  obtained by quotienting by $\usig$.

The fixed points of $B_{C_2}\Sigma_2$ are 
\[
     (B_{C_2}\Sigma_2)^{C_2} \iso S\big( \sigS^\infty\big)/\Sigma_2
     \amalg S\big( (\sigC\otimes\sigS)^\infty\big)/\Sigma_2 
     \iso B\Sigma_2 \amalg B\Sigma_2.
\]
{In the following, we will}
 restrict to the summand $S\big( \sigS^\infty\big)/\Sigma_2$. We will write $\iota_1\colon S\big( \sigS^\infty\big)/\Sigma_2 \into  (B_{C_2}\Sigma_2)^{C_2}$ for the inclusion.

\begin{prop}
    \label{C2-to-classical-classifying-space}
    The underlying homomorphism 
    \[
        \Res\colon \bH^\bigstar(B_{C_2}{\Sigma_2}) \rtarr \rH^*(B{\Sigma_2})[\usig^{\pm 1}] \iso \F_2[t,\usig^{\pm 1}]
    \]
    and the modified fixed point homomorphism 
    \[
        \mFix\colon \bH^\bigstar(B_{C_2}{\Sigma_2}) \rtarr \rH^*(B{\Sigma_2})^2[\asig^{\pm 1}] \iso (\F_2[t])^2[\asig^{\pm 1}]
    \]
    are given by
    \[
        \Res(c)=\usig t, \qquad \Res(d)=\usig t^2, \qquad
        \mFix(d)=\asig(t,t).
    \]
\end{prop}

\begin{rmk}
    We have not stated the value of $\mFix(c)$ in \cref{C2-to-classical-classifying-space}. The ambiguity in the definition of the class $c$ (\cref{cAmbiguous}) means that the value $\mFix(c)$ is not well defined. Depending on the choice for $c$, the value $\mFix(c)$ is either 0 or $\asig(1,1)$. To record this ambiguity, we will denote $\mFix (c) = \asig (\epsilon, \epsilon),$ where $\epsilon \in \{0, 1\},$ 
    {when it appears in the proof of \cref{GeomFixedA} below.}
\end{rmk}

It follows that the restriction of $\mFix(d)$ to the summand 
$S\big( \sigS^\infty\big)/\Sigma_2$ is $\iota_1^* \mFix(d) = \asig t$.

\begin{prop}
\label{UnderlyingA}
    The underlying homomorphism $\Res\colon \cAC_\bigstar \to \cAcl_*[\usig^{\pm 1}]$ is given by 
    \[
        \Res(\etau_n) = \cxi_{n+1}/\usig^{2^{n}-1}, \qquad 
        \Res({\exi_n}) = \cxi_n^2/\usig^{2^n-1}.
    \]
\end{prop}

The powers of $\usig$ in these formulas appear only to make $\Res$ into a map of $RO(C_2)$-graded rings. The reader is encouraged to ignore the denominators.

\begin{pf}
    The elements $\exi_n$ and $\etau_n$ of $\cAC_\bigstar$ are defined by the equations \cref{c-coaction} and \cref{d-coaction}. Comparing $\Res\hat\psi(d)$ to $\hat\psi(\Res d)=\hat\psi(\usig t^2)=\usig (\hat\psi t)^2$ gives
    \[
        \sum_{n\geq 0} \usig^{2^n} t^{2^{n+1}} \otimes \Res(\exi_n) = 
        \sum_{n\geq 0} \usig t^{2^{n+1}} \otimes \cxi_n^2.
    \]
    This gives the formula for $\Res(\exi_n)$.
    Similarly, comparing $\Res\hat\psi(c)$ to $\hat\psi(\Res c)=\hat\psi(\usig t)=\usig\hat\psi(t)$ gives 
    \[
        \usig t\otimes 1 + \sum_{n \geq 0} \usig^{2^n}t^{2^{n+1}}\otimes \Res(\etau_n) = \sum_{n \geq 0} \usig t^{2^n} \otimes \cxi_n.
    \]
\end{pf}

In other words, the first few values of $\Res$ are $\Res(\etau_0) = \cxi_1$, $\Res(\etau_1) = \frac1{\usig} \cxi_2$, $\Res(\etau_2) = \frac1{\usig^3}\cxi_3$, \mbox{$\Res(\exi_1) = \frac1{\usig} \cxi_1^2$}, $\Res(\exi_2) = \frac1{\usig^3} \cxi_2^2$.
\medskip

We next consider the effects of geometric fixed points on the equivariant dual Steenrod algebra. However, the target of \eqref{defn:mFix} in the case $\bX = \bH$ would be $H_*( \Fix \bH)[\asig^{\pm 1}]$. As we will  want a comparison to the classical dual Steenrod algebra, we will further compose with the projection $\Fix \bH \simeq H[v] \to H$. Thus, this might be called a doubly modified geometric fixed point homomorphism.

\begin{prop}
    \label{GeomFixedA}
    The {(doubly modified)} geometric fixed point homomorphism $\mFix\colon \cAC_\bigstar \to \cAcl_*[\asig^{\pm 1}]$ is given by
    \[
        \mFix(\etau_i) = 0, \qquad \mFix(\exi_i) = \cxi_i/\asig^{2^i-1}.
    \]
\end{prop}

\begin{proof}
    The elements $\exi_n$ and $\etau_n$ of $\cAC_\bigstar$ are defined by the equations \cref{c-coaction} and \cref{d-coaction}. 
    Recall that $\iota_1$ denotes the inclusion $B\Sigma_2 \iso S\big( \sigS^\infty\big)/\Sigma_2 \into  (B_{C_2}\Sigma_2)^{C_2}$.
    Comparing $\iota_1^* \mFix \hat \psi (d)$ to $\hat \psi (\iota_1^*\mFix d) = \hat \psi (\asig t) = \asig \hat \psi (t)$ gives
    \[
    \sum_{n \geq 0} \asig^{2^n}  t^{2^n} \otimes \mFix (\exi_n) = \sum_{n \geq 0} \asig t^{2^n} \otimes \cxi_n.
    \]
    This gives the formula for $\mFix (\exi_n).$
    
    Similarly, comparing $\iota_1^*\mFix (\hat \psi (c))$ to $\hat \psi (\iota_1^*\mFix c) = \hat \psi (\asig  \epsilon) = \asig \hat \psi  \epsilon$ gives
    \[
    \asig  \epsilon \otimes 1 + \sum_{n \geq 0} \asig^{2^n}  t^{2^n} \otimes \mFix (\etau_n) = \asig \epsilon \otimes 1.
    \] 
    This gives the formula for $\mFix (\etau_n)$. Note that the value of $\mFix (\etau_i)$ does not depend on the choice of $c$ giving rise to the value of $\epsilon.$
\end{proof}

\section{Equivariant Dyer-Lashof operations}

Dyer-Lashof operations acting on the homology of $C_2$-equivariant $E_\infty$-spaces were introduced in \cite{Wilson17,Wilson19}. For any $C_2$-representation $V$, there is a corresponding little $V$-disks operad $E_V$. Dyer-Lashof operations in the homology of $E_\rho$-spaces were considered in \cite{BW}, where $\rho$ is the regular representation of $C_2$. In particular, Behrens-Wilson define operations
\[
    \bH_{k\rho+1}X \xrtarr{Q^{k\rho+\sigma}} \bH_{(2k+1)\rho}X
    \qquad \text{and}
    \qquad
    \bH_{k\rho+1}X \xrtarr{Q^{(k+1)\rho}} \bH_{(2+1)k\rho+1}X.
\]
We will denote these same operations using subscript notation,
as
\begin{equation}
\label{eq:Q0Q1krho+1}    
    \bH_{k\rho+1}X \xrtarr{Q_0} \bH_{(2k+1)\rho}X
    \qquad \text{and}
    \qquad
    \bH_{k\rho+1}X \xrtarr{Q_1} \bH_{(2+1)k\rho+1}X.
\end{equation}
The same approach defines Dyer-Lashof operations on the homology of any $E_\rho$-algebra $\bX$ in $C_2$-spectra. Behrens-Wilson focus on the case of $\bX = \Sigma^\infty_{C_2} X_+$ for $X$ an $E_\rho$-algebra in $C_2$-spaces.

The method of \cite{BW} can also be used to define operations
\[
    \bH_{k\rho}\bX \xrtarr{Q^{k\rho}} \bH_{2k\rho}\bX
    \qquad \text{and}
    \qquad
    \bH_{k\rho}\bX \xrtarr{Q^{k\rho+\sigma}} \bH_{2k\rho+\sigma}\bX,
\]
which we will rewrite in subscript notation as 
\begin{equation}
\label{eq:Q0Q1krho}    
    \bH_{k\rho}\bX \xrtarr{Q_0} \bH_{2k\rho}\bX
    \qquad \text{and}
    \qquad
    \bH_{k\rho}\bX \xrtarr{Q_1} \bH_{2k\rho+\sigma}\bX.
\end{equation}
To define these, represent an element $x \in \bH_{k\rho} \bX$ as a map $S^{k\rho} \to \bH \wedge \bX$. 
We may then form the composition
\begin{equation}    
\label{eq:DefineQkrho}
    \tilde{x}\colon \tMC \Ptwo_\rho(S^{k\rho}) \xrightarrow{\Ptwo_\rho(x)} \bMC\Ptwo_\rho (\bH\wedge \bX) \to \bMC(\bH\wedge \bX) \to \bMC \bX,
\end{equation}
where $\Ptwo_\rho(S^{k\rho})$ is the extended power $\mathcal{C}_\rho (2)_+ \wedge_{\Sigma_2} S^{k \rhoC \otimes \rhoS}$ for $C_\rho$ an $E_\rho$-operad, the second map is induced by the $E_\rho$-structure of $\bH\wedge \bX$, and the third map is the multiplication $\bH \wedge \bH \to \bH$.

\begin{prop}
\label{prop:CoactrhoSymP}
The {reduced} homology of $\Ptwo_\rho(S^{k\rho})$ is given by
\[
    \tMC \Ptwo_\rho(S^{k\rho}) \cong \bMC \{f_{2k\rho}, \, f_{2k \rho + \sigma} \},
\]
with trivial right $\cAC_\bigstar$-coaction. In other words,
\[
    \psi_R(f_{2k\rho}) = f_{2k\rho} \otimes 1,
    \qquad 
    \psi_R(f_{2k\rho+\sigma}) =f_{2k\rho+\sigma}\otimes 1.
\]
\end{prop}

\begin{pf}
    The homology statement follows as in \cite[Theorem~2.15]{Wilson17}, given that $\Ptwo_\rho(S^{k\rho})$ is the stage $F_1 \mathbb{P}_2 S^{k\rho}$ of the filtration considered in \cite{Wilson17}.

    The coaction on $f_{2k\rho+\sigma}$ is necessarily trivial for degree reasons, as $\cAC_\bigstar$ vanishes in degree $\sigma$. The coaction on $f_{2k\rho}$ is necessarily of the form
    \[
        \psi_R(f_{2k\rho}) = f_{2k\rho} \otimes 1 + \varepsilon f_{2k\rho+\sigma} \otimes \asig  = (f_{2k\rho} + \varepsilon \asig f_{2k\rho+\sigma}) \otimes 1
    \]
    for $\varepsilon\in \{0,1\}$.
 But the counit axiom for the $\cAC_\bigstar$-coaction forces $\varepsilon$ to be 0.
\end{pf}

We can now use \eqref{eq:DefineQkrho} to define Dyer-Lashof operations on classes in degree $k\rho$.

\begin{defn}
\label{defn:DLErho}
    Let $\bX$ be an $E_\rho$-algebra in $C_2$-spectra.
    Given a class $x\in \bH_{k\rho}\bX$, we define the elements 
    \[
        Q_0(x) = Q^{k\rho}(x) \in \bH_{2k\rho}\bX
        \qquad \text{and} \qquad
        Q_1(x) = Q^{k\rho+\sigma}(x) \in \bH_{2k\rho+\sigma}\bX
    \]
    as 
    \[
        Q_0(x) = Q^{k\rho}(x) = \tilde{x}_*(f_{2k\rho}) \qquad \text{and} \qquad Q_1(x)= Q^{k\rho+\sigma}(x) = \tilde{x}_*(f_{2k\rho+\sigma}),
    \]
    where $\tilde{x}$ is the composition defined in \eqref{eq:DefineQkrho}.
\end{defn}

\subsection{Comparison to Wilson's stable operations}
Now if $\bX$ is an $E_\infty$-algebra, there are two definitions of operations $Q_0$ and $Q_1$ in the homology of $X$: those defined above and the stable operations of \cite{Wilson17}. We show that these coincide.

\begin{prop}
\label{StableOpsComp}
    The operations $Q_0$ and $Q_1$ on the homology of $E_\rho$-algebras
    agree with the corresponding stable operations defined in \cite{Wilson17,Wilson19}.
\end{prop}

\begin{proof}
    Recall that we have only defined $Q_0$ and $Q_1$ on classes in degrees $k\rho$ or $k\rho+1$.

    For classes in degree $k\rho$, this relies on an analysis of the map on homology induced by $\Ptwo_\rho(S^{k\rho}) \to \Ptwo_\infty(S^{k\rho})$. According to \cite[Theorem~2.15]{Wilson17}, this is the inclusion
    \[
        \bH_\bigstar \{ e_{2k\rho},e_{2k\rho+\sigma}\} \into \bH_\bigstar\{e_{2k\rho},e_{2k\rho+\sigma}, e_{(2k+1)\rho},e_{(2k+1)\rho+\sigma}, \dots \}.
    \]
    It follows that the stable $Q_0$ and $Q_1$ of \cite{Wilson17} on classes in degree $k\rho$ agree with those of \cref{defn:DLErho}.

    Similarly, for classes of degree $k\rho+1$, we consider $\Ptwo_\rho(S^{k\rho+1}) \to \Ptwo_\infty(S^{k\rho+1})$. \cite[Proposition~3.3]{BW} gives a computation
    \[
        \tMC \Ptwo_\rho S^{k\rho+1} \iso \bMC \{ e_{(2k+1)\rho},e_{(2k+1)\rho+1}\},
    \]
    while \cite[Proposition~2.4.1]{Wilson19} states
    \[
        \tMC \Ptwo_\infty S^{k\rho+1} \iso \bMC \{ e_{(2k+1)\rho},e_{(2k+1)\rho+1},
        e_{(2k+2)\rho},e_{(2k+2)\rho+1},
        \dots
        \}.
    \]
    The argument for this is similar to the proof of \cite[Proposition~3.3]{BW}. Namely, the analogue of \cite[(3.5)]{BW} is a cofiber sequence
    \begin{equation}
        \label{eq:cofibPtwoinfSkrho}
        \Sigma^1 \Ptwo_\infty(S^{k\rho}) \to \Ptwo_\infty(S^{k\rho+1}) \to S(\infty\rho)_+ \wedge S^{2k\rho+2} \simeq S^{2k\rho+2}.
    \end{equation}
    The connecting homomorphism in homology is the map
    \[
        \bH_\bigstar \{ f_{2k\rho+2} \} \to \bH_\bigstar\{\Sigma^2 e_{2k\rho},\Sigma^2 e_{2k\rho+\sigma}, \Sigma^2 e_{(2k+1)\rho},\Sigma^2 e_{(2k+1)\rho+\sigma}, \dots \}
    \]
    given by
    \[
        f_{2k\rho+2} \mapsto \Sigma^2 e_{2k\rho}.
    \]
    This follows from the fact that the underlying space of $\Ptwo_\infty(S^{k\rho+1})$, which is $\Sigma^{2k+1} \R\bP_{2k+1}^\infty$, has no homology in degree $4k+1$.

    Then the inclusion $S(\rhoC\otimes \sigS) \into S(\infty \rhoC\otimes \sigS)$ induces a map of cofiber sequences from \cite[(3.5)]{BW} to \eqref{eq:cofibPtwoinfSkrho} sending the generators in $\tMC \Ptwo_\rho S^{k\rho+1}$ to the generators of the same name in $\tMC\Ptwo_\infty S^{k\rho+1}$. It follows that the stable $Q_0$ and $Q_1$ of \cite{Wilson17} on classes in degree $k\rho+1$ agree with those of \cite{BW} for $E_\rho$-algebras.
\end{proof}

\begin{rmk}    
    We note that, just as is the case non-equivariantly, Dyer-Lashof operations on equivariant {\it finite} loop spaces are in general not additive. In contrast, the stable operations for $E_\infty$-algebras are additive.
\end{rmk}

\subsection{The Cartan formula}

As we will see in
\cref{hlgyOmRhoSRho1}, 
the example of interest in this article will be a homology algebra that can be identified, as an $E_\rho$-algebra, with the underlying $E_\rho$-algebra of an $E_\infty$-algebra. This means that the equivariant Dyer-Lashof operations, which in general are not even additive, will be particularly well-behaved. For instance, they will inherit the $E_\infty$ form of the Cartan formula established in \cite{Wilson19}.

We will express the Cartan formula of \cite[Corollary 1.3.2]{Wilson19} in the subscript notation. For this, it is convenient to say that $x \equiv \varepsilon \pmod{\rho}$, for $\varepsilon\in\{0,1\}$, if the degree of $x$ is equal to $\varepsilon$ plus a multiple of $\rho$. Then the Cartan formula reads as follows.

\begin{thm}[{\cite[Corollary~1.3.2]{Wilson19}}]  
Let $\bX$ be an $E_\rho$-algebra underlying a $C_2-E_\infty$-algebra and let $x,y\in \bMC \bX$ lie in degrees congruent to either 0 or 1 modulo $\rho$. Then there are Cartan formulas
\[
Q_0(x\otimes y) = 
\begin{cases}
    Q_0(x)\otimes Q_0(y) & x,y\equiv0\pmod{\rho} \\
    Q_0(x)\otimes Q_0(y) + \asig Q_1(x) \otimes Q_0(y) & x\equiv0, y\equiv1\pmod{\rho} \\
    Q_0(x)\otimes Q_0(y) + \asig Q_0(x) \otimes Q_1(y) & x\equiv1, y\equiv0\pmod{\rho} \\
\end{cases}
\]
and
\[
Q_1(x\otimes y) = 
\begin{cases}
    Q_1(x) \otimes Q_0(y) + Q_0(x) \otimes Q_1(y)  & x,y\equiv0\pmod{\rho} \\
    \qquad + \asig Q_1(x) \otimes Q_1(y) & \\
    Q_0(x) \otimes Q_1 (y) + \usig Q_1(x)\otimes Q_0(y) & x\equiv0, y\equiv 1 \pmod{\rho} \\
    Q_1(x) \otimes Q_0 (y) + \usig Q_0(x)\otimes Q_1(y) & x\equiv1, y\equiv 0 \pmod{\rho}. \\\end{cases}
\]
\end{thm}

\section{The equivariant Nishida relations}
\label{sec:Nishida}

\newcommand{\SymPowerGenDeg}{\delta}

Here we will establish the $C_2$-equivariant analogue of Baker's formulation \cref{eq:classicalcoNishida} of the co-Nishida relations, which describe the $\cAC_\bigstar$-coaction on the output of Dyer-Lashof operations.
Wilson states similar formulas in the stable case in \cite[Section~5.1]{Wilson17}.

For any $E_\rho$-algebra $\mathbf{Y}$ in $C_2$-spectra, let us write $\Ptwo_\rho \mathbf{Y} \xrightarrow{\alpha} \mathbf{Y}$ for the resulting multiplication.
Recall that for $V$ a $C_2$-representation, $\bX$ an $E_\rho$-algebra,  $e\in \bH_\SymPowerGenDeg \Ptwo_\rho(S^V)$,
and $x\in \bH_V \bX$, we define a class 
\[
	\Theta^e (x) := \tilde{x}_*(e) \in \bH_\SymPowerGenDeg \bX
\]
as the image of $e$ in the composite
\[
	\begin{tikzcd}[row sep={1ex}]
		\bH_\SymPowerGenDeg \Ptwo_\rho S^V \ar[r, "\Ptwo_\rho(x)"]  \ar[rrr, bend left=4ex, "\tilde{x}_*",start anchor={north}] & 
        \bH_\SymPowerGenDeg \Ptwo_\rho \bH \wedge \bX  \ar[r, "\alpha_{\bH\wedge X_+}"] & \bH_\SymPowerGenDeg \bH \wedge \bX \ar[r, "\mu"] & \bH_\SymPowerGenDeg \bX \\
		e \ar[rrr, |->] & & & \Theta^e(x)=\tilde{x}_*(e).
	\end{tikzcd}
\]
where the last map is induced from the multiplication $\mu\colon \bH \wedge \bH \rtarr \bH$.

We will verify the equivariant version of \cite[(2.6)]{Baker15}, namely:

\begin{thm}[Co-Nishida relation] 
\label{thm:equivarRightCoactionFormula}
Let $\bX$ be an $E_\rho$-algebra in $C_2$-spectra,
let $x\in \bH_V \bX$, and let $e\in \bH_\SymPowerGenDeg \Ptwo_\rho(S^V)$. If $\psi_R(e)=\sum_i e_i \otimes a_i$, then
\[
	\psi_R \Theta^e(x) = \sum_i \Theta^{e_i} \psi_R(x) \cdot (1\otimes a_i).
\]
\end{thm}

To parse this formula, note that 
\begin{enumerate}
\item $\bH$ is an equivariant $E_\infty$ ring and in particular an $E_\rho$-ring, so that $\bH \wedge \bX$ is an $E_\rho$-ring and the formula $\Theta^{e_i} \psi_R (x)$ is defined, and 
\item $\bMC (\bH \wedge \bX) \iso \bMC \bX \otimes_{\bMC} \cAC_\bigstar$ is a ring, so that it makes sense to multiply the two elements $\Theta^{e_i} \psi_R(x)$ and $1\otimes a_i$.
\end{enumerate}

\begin{pf}
As in \cite[Section 2]{Baker15}, we establish the formula by considering a large diagram. We will often suppress factors such as $1\wedge$ in a morphism, in order to avoid clutter. For instance, the first horizontal morphism should more properly be written as $1\wedge \Ptwo_\rho x$. 
There are many instances of the symbol $\bH$, and we use color to help distinguish between these. In some instances, we use a combination of colors on a single $\bH$ to denote a multiplication. For instance, we write
${\color{gray}\bH} \wedge {\color{blue}\bH} \wedge {\color{red}\bH} \to \Htwo{gray}{blue} \wedge {\color{red}\bH}$ to denote which copies of $\bH$ have been multiplied.

\adjustbox{scale=0.75,center}{%
\begin{tikzcd}
{\color{red}\bH} \wedge \Ptwo_\rho S^V \ar[d,"{\color{gray}\eta}"] \ar[r,"\Ptwo_\rho x"] 
&
{\color{red}\bH} \wedge \Ptwo_\rho (\bH \wedge \bX) \ar[d,"{\color{gray}\eta}"] \ar[r,"\alpha"]
&
{\color{red}\bH} \wedge  \bH \wedge \bX \ar[d,"{\color{gray}\eta}"] \ar[r,"\mu"]
& 
\Htwo{red}{black} \wedge \bX \ar[d,"{\color{gray}\eta}"]
\\
{\color{gray}\bH} \wedge {\color{red}\bH} \wedge \Ptwo_\rho S^V \ar[r," \Ptwo_\rho x"] \ar[d,"tw"]
&
{\color{gray}\bH} \wedge {\color{red}\bH} \wedge \Ptwo_\rho (\bH \wedge \bX)  \ar[r,"\alpha"] \ar[d,"tw"]
&
{\color{gray}\bH} \wedge {\color{red}\bH} \wedge \bH \wedge \bX \ar[r,"1\wedge \mu \wedge 1"] \ar[d,"tw"]
& 
{\color{gray}\bH} \wedge \Htwo{red}{black} \wedge \bX 
\\
{\color{gray}\bH} \wedge \Ptwo_\rho S^V \wedge {\color{red}\bH} \ar[r,"\Ptwo_\rho x"] \ar[dr,"\Ptwo_\rho(\psi_R x)" swap]
&
{\color{gray}\bH} \wedge \Ptwo_\rho (\bH \wedge \bX) \wedge {\color{red}\bH}  \ar[r,"\alpha"] \ar[d,"\Ptwo_\rho {\color{blue}\eta}"]
&
{\color{gray}\bH} \wedge \bH \wedge \bX \wedge {\color{red}\bH} \ar[r,equal] \ar[d," {\color{blue}\eta}"] 
& 
{\color{gray}\bH} \wedge \bH \wedge \bX \wedge {\color{red}\bH}\ar[u,"1\wedge \mu"]
\\
&
{\color{gray}\bH} \wedge \Ptwo_\rho({\color{blue}\bH} \wedge \bH \wedge \bX) \wedge {\color{red}\bH} \ar[r,"\alpha"]
&
{\color{gray}\bH} \wedge {\color{blue}\bH} \wedge \bH \wedge \bX \wedge {\color{red}\bH} \ar[r,"\mu\wedge 1"]
& 
\Htwo{gray}{blue}
\wedge \bH \wedge \bX \wedge {\color{red}\bH} \ar[u,equal]
\end{tikzcd}
}

Passing to degree $\SymPowerGenDeg$ homotopy then produces the commuting diagram of homology groups

\adjustbox{scale=0.75,center}{%
\begin{tikzcd}
{\color{red}\tilde{\bH}_\SymPowerGenDeg} \Ptwo_\rho S^V \ar[d,"{\color{gray}\eta}"] \ar[r,"{\color{red}\bH_\SymPowerGenDeg}\Ptwo_\rho x"] 
&
{\color{red}\bH_\SymPowerGenDeg} \Ptwo_\rho (\bH \wedge \bX) \ar[d,"{\color{gray}\eta}"] \ar[r,"{\color{red}\bH_\SymPowerGenDeg}\alpha"]
&
{\color{red}\bH_\SymPowerGenDeg} ( \bH \wedge \bX) \ar[d,"{\color{gray}\eta}"] \ar[r,"\mu"]
& 
\Htwo{red}{black}_\SymPowerGenDeg \bX \ar[d,"{\color{gray}\eta}"]
\\
{\color{gray}\bH_\SymPowerGenDeg} ({\color{red}\bH} \wedge \Ptwo_\rho S^V) \ar[r,"{\color{gray}\bH_\SymPowerGenDeg} \Ptwo_\rho x"] \ar[d,"{\color{gray}\bH_\SymPowerGenDeg}tw"]
&
{\color{gray}\bH_\SymPowerGenDeg} ( {\color{red}\bH} \wedge \Ptwo_\rho (\bH \wedge \bX))  \ar[r,"{\color{gray}\bH_\SymPowerGenDeg} \alpha"] \ar[d,"{\color{gray}\bH_\SymPowerGenDeg} tw"]
&
{\color{gray}\bH_\SymPowerGenDeg} ( {\color{red}\bH} \wedge \bH \wedge \bX ) \ar[r,"{\color{gray}\bH_\SymPowerGenDeg} \mu \wedge 1"] \ar[d,"{\color{gray}\bH_\SymPowerGenDeg} tw"]
& 
{\color{gray}\bH_\SymPowerGenDeg} ( \Htwo{red}{black} \wedge \bX )
\\
{\color{gray}\bH_\SymPowerGenDeg} ( \Ptwo_\rho S^V \wedge {\color{red}\bH}) \ar[r,"{\color{gray}\bH_\SymPowerGenDeg}\Ptwo_\rho x"] \ar[dr,"{\color{gray}\bH_\SymPowerGenDeg} \Ptwo_\rho(\psi_R x)" swap]
&
{\color{gray}\bH_\SymPowerGenDeg} ( \Ptwo_\rho (\bH \wedge \bX) \wedge {\color{red}\bH} ) \ar[r,"{\color{gray}\bH_\SymPowerGenDeg} \alpha"] \ar[d,"{\color{gray}\bH_\SymPowerGenDeg} \Ptwo_\rho {\color{blue}\eta}"]
&
{\color{gray}\bH_\SymPowerGenDeg} ( \bH \wedge \bX \wedge {\color{red}\bH}) \ar[r,equal] \ar[d,"{\color{gray}\bH_\SymPowerGenDeg} {\color{blue}\eta}"] 
& 
{\color{gray}\bH_\SymPowerGenDeg} ( \bH \wedge \bX \wedge {\color{red}\bH}) \ar[u,"{\color{gray}\bH_\SymPowerGenDeg} \mu"]
\\
&
{\color{gray}\bH_\SymPowerGenDeg } ( \Ptwo_\rho({\color{blue}\bH} \wedge \bH \wedge \bX) \wedge {\color{red}\bH} ) \ar[r,"{\color{gray}\bH_\SymPowerGenDeg} \alpha"]
&
{\color{gray}\bH_\SymPowerGenDeg} ( {\color{blue}\bH} \wedge \bH \wedge \bX \wedge {\color{red}\bH} ) \ar[r,"\mu"]
& 
\Htwo{gray}{blue}_\SymPowerGenDeg
( \bH \wedge \bX \wedge {\color{red}\bH} ) \ar[u,equal]
\end{tikzcd}
}

Tracing the element $e\in {\color{red}\tilde{\bH}_\SymPowerGenDeg} \Ptwo_\rho S^V$ around the diagram gives

\adjustbox{scale=0.9,center}{%
\begin{tikzcd}[column sep={4em}]
e \ar[rrr,|->] \ar[d,|->]
& & & 
\Theta^e x \ar[d,|->,xshift={-7ex},end anchor={[xshift=-1ex]}] 
\qquad \qquad 
\\
\psi_R e = \sum_i e_i \otimes a_i \ar[drrr,bend right=10,|->]
& & & 
\psi_R \Theta^e x = 
\sum_i \Theta^{e_i}\psi_R x \cdot (1\otimes a_i)
\\
& & &
\quad \sum_i (\Theta^{e_i}\psi_R x) \otimes a_i \ar[u,|->]
\end{tikzcd}
}

\end{pf}

\begin{eg}
\label{eg:CoactQzero}
We give an illustration of \cref{thm:equivarRightCoactionFormula}. 
Let $\bX$ be an $E_\rho$-algebra with $x\in \bH_{k\rho}\bX$. According to \cref{defn:DLErho}, we have $Q_0(x) = \Theta^{f_{2k\rho}}(x)$, where $f_{2k\rho} \in \tilde \bH_{2k\rho} \Ptwo_\rho S^{k\rho}$ is primitive according to \cref{prop:CoactrhoSymP}. The co-Nishida relation then says
\[
    \psi_R Q_0(x) = \psi_R \Theta^{f_{2k\rho}}(x) = \Theta^{f_{2k\rho}} \psi_R(x) = Q_0 \psi_R(x).
\]
Similarly, 
\[
    \psi_R Q_1(x) = \psi_R \Theta^{f_{2k\rho+\sigma}}(x) = \Theta^{f_{2k\rho+\sigma}} \psi_R(x) = Q_1 \psi_R(x).
\]
\end{eg}

\section{The $\cAC_\bigstar$-comodule $\bH_\bigstar \Omega^\rho S{^{\rho+1}}$}

\label{sec:ComoduleStructure}

Here, we describe the right $\cAC_\bigstar$-comodule structure on $\bH_\bigstar \Omega^\rho S{^{\rho+1}}$. This is the $RO(C_2)$-graded homotopy of $\bH \wedge \Omega^\rho {S^{\rho+1}}_+$.

\begin{notn}
    We will write $\OmRhoSRho$ for $(\Omega^\rho S^{\rho+1})_+$. That is, this is the space of based $\rho$-loops in $S^{\rho+1}$, but with a disjoint basepoint attached to the loop space.
\end{notn}

Nonequivariantly, the $\mathcal{A}_*$-comodule structure on $H_*\Omega^2S^3$ may be determined via the use of (co-)Nishida relations, as discussed in \cref{AcomodReview}. In particular, the $E_2$-algebra structure on $\Omega^2S^3$ allows one to express $H_*\Omega^2S^3$ as free over the Dyer--Lashof algebra on a single generator (see \cite[Example 1.5.7, Theorem 1.5.11]{Law}). The Nishida relations then allow one to deduce the coaction on all of $H_*\Omega^2S^3$ from that on the generator.

We will mimic this approach in the equivariant setting by using the equivalence
\[
    \bH\wedge \OmRhoSRho \simeq \bH\wedge \bH
\]
of Behrens--Wilson \cite{BW}, along with the determination of the action of equivariant Dyer--Lashof operations on the equivariant dual Steenrod algebra $\cAC_\bigstar$ by Wilson \cite{Wilson19}. We begin with the latter.

\begin{prop}[{\cite[Corollary~1.6.4]{Wilson19}}]
\label{DLetauexi}
\[
    Q_0(\etau_k) = \exi_{k+1}, \qquad \qquad
    Q_1(\etau_k) = \etau_{k+1} + \etau_0 \exi_{k+1},
\]
and
\[
    Q_0(\overline{\etau_k}) = \overline\exi_{k+1}, \qquad \qquad
    Q_1(\overline{\etau_k}) = \overline\etau_{k+1}.
\]    
\end{prop}

\begin{prop}[{\cite[Theorem~4.1]{BW}}]
\label{hlgyOmRhoSRho1}
    There is an isomorphism
    of $H_\bigstar$-algebras 
    \[
        \bMC \Omega^\rho S^{\rho+1}=\bMC[t_n,e_k\ |n\ge0,k\ge1]/(t_n^2=\asig t_{n+1}+\usig e_{n+1})
    \]
    where $|e_k|=(2^k-1)\rho$ and $|t_n|=(2^n-1)\rho+1$. The Dyer--Lashof operations satisfy
\begin{align*}
    Q_0(e_k)&=e_k^2&Q_1(e_k)&=0\\
    Q_0(t_n)&=e_{n+1}&Q_1(t_n)&=t_{n+1}
\end{align*}
\end{prop}

\begin{proof}
    We claim there is an equivalence of $E_\rho$-$\bH$-algebras
\[
    \bH \wedge \OmRhoSRho
\simeq \bH \wedge \bH.
\]
Indeed, following \cite[Section 5]{BW}, let
\[
    \mathrm{Free}_{E_\rho,\mathbf{H}}^*:\mathrm{Alg}_{E_0}(\mathrm{Mod}_{\mathbf{H}})\to \mathrm{Alg}_{E_\rho}(\mathrm{Mod}_{\mathbf{H}})
\]
denote the left adjoint to the forgetful functor. It follows from the proof of \cite[Theorem 5.1]{BW} that $\mathrm{Free}_{E_\rho,\bH}^*(\bH\wedge S^1_+)=\mathbf{H}\wedge (\OmRhoSRho)$. By adjunction, the class $\etau_0\in\pi_1(\bH\wedge \bH)$ therefore provides a map of $E_\rho$-algebras
\[
    \bH\wedge \OmRhoSRho \to \bH\wedge \bH
\]
sending $t_0\mapsto \etau_0$. The proof of \cite[Theorem 1.2]{BW} shows this map is an equivalence, proving the claim. The claimed isomorphism of $\bMC$-algebras then follows from the Hu--Kriz computation of $\bMC \bH$ \cite[Section 6]{HK} (see also \cite[Section 2]{LSWX}) by transferring $
\overline{\exi_k}$ and $\overline{\etau_n}$ along this isomorphism to define $e_k$ and $t_n$.

Moreover, we may prove the claims about Dyer--Lashof operations on the classes $\overline{\exi_k}$ and $\overline{\etau_n}$ in $\bMC \bH$. The class $\overline{\exi_k}$ is the Hopf algebroid conjugate of the class $\exi_k$. By \cref{lemma:xicomesfromMUR}, the class $\exi_k$ is in the image of the map of Hopf algebroids ${MU_\R}_\bigstar MU_\R\to \bH_\bigstar \bH= \cAC_\bigstar$, hence so is its conjugate $\overline{\exi_k}$. In particular, the class $\overline{\exi_k}$
is in the image of the map of $C_2$-$E_\infty$-$\bH$-algebras
\[
    \bH\wedge MU_\R\to \bH\wedge \bH,
\]
and we claim that if $x\in \bH_{k\rho}MU_\mathbb{R}$, then $Q_0(x)=x^2$ and $Q_1(x)=0$. For this, note that $\underline{\bH}_{k\rho}MU_\mathbb{R}$ is a sum of $\underline{\F_2}$'s, as a Mackey functor, and $\bH_{k\rho+\sigma}MU_\mathbb{R}=0$ as follows from \cite[Proposition 2.8]{LSWX}. In particular, the restriction map is injective in these degrees, so the claim about $Q_0$ follows from the fact that $Q_0(x)$ restricts to $Q_0(\mathrm{res}(x))=\mathrm{res}(x)^2$, and the claim about $Q_1$ on $e_k$ follows for degree reasons. The remaining two equations 
{are given in \cref{DLetauexi}}.
\end{proof}

\begin{rmk}
    Although we use the Behrens--Wilson computation of $\bMC \Omega^\rho S^{\rho+1}$ \cite[Theorem~4.1]{BW} above, our generators $e_i$ and $t_j$ are defined differently from those of the same names in Behrens--Wilson. Indeed, Behrens--Wilson define $e_i$ as the norm of $x_i$, where $x_1\in H_1\Omega^2 S^3$ is the fundamental class and $x_i=Q_1x_{i-1}$, and they define $t_0\in \bH_1\Omega^\rho S^{\rho+1}$ as the fundamental class and $t_j=Q_1t_{j-1}$. This allows them to establish the equivalence 
    \[
    \bH\wedge \OmRhoSRho \simeq \bH\wedge \bH
\]
we use above, which we then use to redefine the $e_i$'s
 and $t_j$'s as the image of the $
\overline{\exi_i}$'s and $\overline{\etau_j}$'s, respectively along this isomorphism. However, it is straightforward to show that these two definitions coincide using the above values of $Q_1$ and the fact that the operation $Q_0$ agrees with the norm on the classes $\overline{\etau_j}$'s. This latter fact follows from the fact the restriction map is injective in the degree of $Q_0\overline{\etau_j}$.
\end{rmk}

We will need a description of $\bMC \Omega^\rho S^{\rho + 1}$ as an $\cEC_\bigstar$-comodule. 
In fact, we can do better and describe $\bMC \Omega^\rho S^{\rho + 1}$ as an $\cAC_\bigstar$-comodule. According to \cref{thm:equivarRightCoactionFormula}, the $\cAC_\bigstar$-coaction on $\bMC \Omega^\rho S^{\rho + 1}$ can be computed inductively given the coaction on the {homology of the} $E_\rho$-extended powers $\Ptwo_\rho(S^{k\rho})$
and $\Ptwo_\rho(S^{k\rho-\sigma})$. The former coaction was described \cref{prop:CoactrhoSymP}.
We now establish the $\cAC_\bigstar$-coaction on the homology of $\Ptwo_\rho(S^{k\rho-\sigma})$.

\begin{lem}
    \label{lem:FixPtsSymrhominussig}
    The fixed points of $\Ptwo_\rho(S^{k\rho-\sigma})$ can be identified as
    \[
        \Ptwo_\rho(S^{k\rho-\sigma})^{C_2} \iso S^{2k-1} \vee S^{2k}.
    \]
\end{lem}

\begin{pf}
    The extended power $\Ptwo_\rho(S^{k\rho-\sigma})$ is the Thom space of the bundle $S(\rhoC \otimes \sigS) \times_{\Sigma_2} ((k\rhoC-\sigC) \otimes \rhoS)$ over $S(\rhoC\otimes \sigS)/\Sigma_2$. Suppose that 
    \[
        (\mathbf{x},\mathbf{y}) \in S(\rhoC \otimes \sigS)\times (k\rhoC-\sigC)\otimes \rhoS
    \]
    becomes fixed by $C_2$ after passage to $\Sigma_2$-orbits. Let use write $\gamma$ and $\usig$ for the generators of $C_2$ and $\Sigma_2$, respectively. 

    {\bf Case 1:} The pair $(\mathbf{x},\mathbf{y})$ is already fixed by $\gamma$ before passage to $\Sigma_2$-orbits. Then $\mathbf{x}$ lies in $S(\mathbf{1}\otimes \sigS)\iso \Sigma_2/e$ and $\mathbf{y}$ lies in $k\mathbf{1}\otimes \rhoS$. In other words, after passage to $\Sigma_2$-orbits, the space of all such $(\mathbf{x},\mathbf{y})$ is 
    \[
        S(\mathbf{1}\otimes \sigS) \times_{\Sigma_2} (k\mathbf{1}\otimes \rhoS) \iso \R^{2k},
    \]
    which Thomifies to $S^{2k}$.

    {\bf Case 2:} The pair $(\mathbf{x},\mathbf{y})$ is not fixed by $\gamma$ before passage to $\Sigma_2$-orbits. Then 
    \[
        (\gamma\mathbf{x},\gamma\mathbf{y}) = (-\mathbf{x},\usig\mathbf{y})
    \]
    In other words, $\mathbf{x}$ lies in $S(\sigC\otimes \sigS)\iso (C_2\times\Sigma_2)/\Delta$ and $\mathbf{y}$ is fixed by $\Delta\leq C_2\times \Sigma_2$. But the $\Delta$-fixed points of $(k\rhoC-\sigC) \otimes \rhoS$ has dimension $2k-1$. Thus, after passage to $\Sigma_2$-orbits, the space of all such $(\mathbf{x},\mathbf{y})$ is 
    \[
        S(\sigC\otimes \sigS) \times_{\Sigma_2} ((k\rhoC-\sigC) \otimes \rhoS)^\Delta \iso (C_2\times \Sigma_2)/\Delta \times_{\Sigma_2} \R^{2k-1} \iso \R^{2k-1}.
    \]
    Thus, these fixed points constitute a disjoint copy of $S^{2k-1}$ in $\Ptwo_\rho(S^{k\rho-\sigma})^{C_2}$.
\end{pf}

\begin{prop}
    The right $\cAC_\bigstar$-coaction on 
    \[
        \tMC (\Ptwo_\rho (S^{k\rho-\sigma})) \iso \bMC \{e_{2k\rho-\sigma-1},e_{2k\rho-\sigma}\}
    \]
    is given by
    \begin{align*}
        \psi_R(e_{2k\rho-\sigma-1})&=e_{2k\rho-\sigma-1}\otimes1\\
        \psi_R(e_{2k\rho-\sigma})&=
            e_{2k\rho-\sigma}\otimes1+e_{2k\rho-\sigma-1}\otimes\etau_0 
    \end{align*}
\end{prop}

\begin{proof}
    The homology was computed in \cite[Proposition~3.3]{BW}.
    The bottom class $e_{2k\rho-\sigma-1}$ is primitive for degree reasons, as $\cAC_{-1}=0$. One has that $\cAC_1=\F_2\{\etau_0, \asig\exi_1 \}$,
    so one has a coaction formula of the form
    \[
    \psi_R(e_{2k\rho-\sigma})=e_{2k\rho-\sigma}\otimes1+\delta\, e_{2k\rho-\sigma-1}\otimes \etau_0 + \varepsilon\, e_{2k\rho-\sigma-1}\otimes \asig \exi_1
    \]
    for some $\delta, \varepsilon\in\F_2$. 
    
    We first show that the coefficient $\delta$ must be $1$. This can be seen by applying  the restriction map, since the underlying space
    \[
        \Res \Ptwo_\rho S^{k\rho-\sigma}\simeq \Sigma^{2k-1} \R\mathbb{P}_{2k-1}^{2k}
    \]
    has a $\xi_1$-term in the coaction on the top cell, and the restriction of $\etau_0$ is $\xi_1$ according to \cref{UnderlyingA}.   
   
    To determine $\varepsilon$, we consider fixed points. \cref{lem:FixPtsSymrhominussig} identifies the fixed points of $\Ptwo_\rho S^{k\rho-\sigma}$ as $S^{2k-1}\vee S^{2k}$.     
    Thus $H_* \Big((\Ptwo_\rho S^{k\rho-\sigma})^{C_2}\Big) \iso H_*\{e_{2k-1},e_{2k}\}$ has a trivial coaction. On the other hand, \cref{GeomFixedA} gives a derivation of the coaction from the $\cAC_\bigstar$-coaction on $ \tMC (\Ptwo_\rho (S^{k\rho-\sigma}))$. That formula gives $\psi_R(e_{2k})= e_{2k}\otimes 1 + \varepsilon e_{2k-1}\otimes \cxi_1$. It follows that $\varepsilon$ is 0.
    
\end{proof}

This calculation allows us to employ the co-Nishida relations on classes in degree $k\rho+1$. 

\begin{cor}
    \label{cor:CoactQdegreekrhoplusone}
    Let $\bX$ be an $E_\rho$-algebra and let $x\in \bH_{k\rho+1}\bX$. Then
    \[
        \psi_R Q_0(x) =Q_0 \psi_R (x)
    \]
    and
    \[
        \psi_R Q_1(x) = Q_1 \psi_R(x) + Q_0 \psi_R(x) \cdot (1\otimes \etau_0).
    \]
\end{cor}

We are now ready to deduce the coaction on $\bMC \Omega^\rho S^{\rho + 1}$.

\begin{prop} \label{prop:A-comod0rhoSrho+1}
    The right $\cAC_\bigstar$-comodule structure on $\bMC \Omega^\rho S^{\rho + 1}$ is given by
    \begin{align*}
        \psi_R(e_{k}) & = \sum_{j=0}^{k-1} e_{k-j}^{2^j} \otimes \exi_{j}, \\
        \psi_R (t_k) & = t_k \otimes 1 + \sum_{j=0}^{k-1} e_{k-j}^{2^j} \otimes \etau_j. 
    \end{align*}
\end{prop}

\begin{proof}
    The formula holds for $t_0$
    {since it is the image of the (primitive) generator in $\bMC S^1 \to \bMC \Omega^\rho S^{\rho+1}$.}
    Assume by induction that the formula holds for $t_k.$ Then the coNishida relations,
   as described in \cref{cor:CoactQdegreekrhoplusone},
   imply
    \begin{align*}
        \psi_R(e_{k+1}) & = Q_0 (\psi_R(t_k)) \\
        & = Q_0 \left(t_k \otimes 1 + \sum_{j=0}^{k-1} e_{k-j}^{2^j} \otimes \etau_j \right) \\
        & = e_{k + 1} \otimes 1 + \sum_{j=0}^{k-1} e_{k-j}^{2^{j + 1}} \otimes \exi_{j + 1} \\
        & = \sum_{j=0}^{k} e_{k+1-j}^{2^j} \otimes \exi_{j},
    \end{align*}
    using the Cartan formula and \Cref{hlgyOmRhoSRho1}. Additionally, the coNishida relations (\cref{thm:equivarRightCoactionFormula}) imply
    \begin{align*}
        \psi_R(t_{k + 1}) & = \psi_R(Q_1 (t_k)) \\
        & = Q_1 \left( t_k \otimes 1 + \sum_{j=0}^{k-1} e_{k-j}^{2^j}\otimes\etau_j \right) + ( 1 \otimes \etau_0) Q_0 \left( t_k \otimes 1 + \sum_{j=0}^{k-1} e_{k-j}^{2^j}\otimes\etau_j \right) \\
        & = t_{k + 1} \otimes 1 + \sum_{j=0}^{k-1} \left( Q_0 (e_{k -j}^{2^j}) \otimes Q_1(\etau_j) \right) \\
        & \qquad \quad + e_{k + 1} \otimes \etau_0 + \sum_{j=0}^{k-1} \left(Q_0(e_{k-j}^{2^j}) \otimes \etau_0 Q_0(\etau_j) \right) \\
        & = t_{k + 1} \otimes 1 + e_{k + 1} \otimes \etau_0 + \sum_{j=0}^{k-1} \left( e_{k-j}^{2^{j + 1}} \otimes (\etau_{j + 1} + \etau_0 \exi_{j + 1}) + e_{k-j}^{2^{j + 1}} \otimes \etau_0 \exi_{j + 1} \right) \\
        & = t_{k + 1} \otimes 1 + \sum_{j=0}^{k} e_{k + 1 - j}^{2^{j}} \otimes \etau_k,
    \end{align*}
    where in the third line we have used the Cartan formula and that $Q_1(e_j)=0.$ In the fourth line we also use \cref{DLetauexi}.
\end{proof}

\subsection{The $\mathcal{A}^h_\bigstar$-comodule $\bH^h_\bigstar \Omega^\rho S{^{\rho+1}}$}

We set
\[
    \bH^h_\bigstar\Omega^\rho S^{\rho+1}:=\pi_\bigstar(F(E{C_2}_+,\bH \wedge \OmRhoSRho))
\]

\begin{proposition}\label{prop:borelcomodule}
    There is an isomorphism
    of $\bH_\bigstar^h$-algebras 
    \[
        \bH^h_\bigstar \Omega^\rho S^{\rho+1}=\bMC^h[t_n,e_k\ |n\ge0,k\ge1]^{\,\,\widehat{}}_{\asig}/(t_n^2=\asig t_{n+1}+\usig e_{n+1})
    \]
    where $|e_k|=(2^k-1)\rho$ and $|t_n|=(2^n-1)\rho+1$. The right $\mathcal{A}^h_\bigstar$-coaction is given by
        \begin{align*}
        \psi_R(e_{k}) & = \sum_{j=0}^{k-1} e_{k-j}^{2^j} \otimes \exi_{j}, \\
        \psi_R (t_k) & = t_k \otimes 1 + \sum_{j=0}^{k-1} e_{k-j}^{2^j} \otimes \etau_j. 
    \end{align*}
\end{proposition}

\begin{proof}
    As in the proof of  \cref{hlgyOmRhoSRho1}, we have an equivalence of ring spectra $\bH\wedge \OmRhoSRho \simeq \bH\wedge\bH$ and hence an equivalence of ring spectra 
    \[
    F(E{C_2}_+,\bH \wedge \OmRhoSRho )\simeq F(E{C_2}_+,\bH\wedge\bH).
    \]
    The claim then follows from the computation of $\mathcal{A}^h_\bigstar$ as well as the map  $\cAC_\bigstar\to\mathcal{A}^h_\bigstar$ as shown by Hu--Kriz \cite[Corollary 6.40 and Theorem 6.41]{HK}, see also \cite[Theorem 2.14]{LSWX}.
\end{proof}

\section{The Snaith Splitting}

We will make use of the fact that loop spaces split stably, so that in particular the homology similarly splits. We start with a reminder of the relevant Snaith splitting.

\begin{thm}[{\cite[Chapter VII, Theorem 5.7]{LewisMayStein86}}] 
\label{thm:Snaith}
    Suppose $G$ is a finite group, $X$ is $G$-connected, and $V$ is a $G$-representation with a trivial summand. Then there is an equivalence of $G$-spectra
    \begin{align} \label{eq:Snaith}
        \Sigma^\infty_G \Omega^V \Sigma^V X_+ \simeq \bigvee_{r \geq 0} \Sigma^\infty_G \Pow^r_V X 
    \end{align}
    where $\Pow^r_V X = C_V(r)_+ \wedge_{\Sigma_r} X^{\wedge r}$ and $C_V(r)$ is the {configuration space of $r$ points} in $V.$
\end{thm}

\begin{cor}
\label{cor:HlgySnaithSplitting}
    With the same assumptions as \cref{thm:Snaith}, we have a splitting of $\cAC_\bigstar$-comodules
    \[
    \bMC \Omega^V \Sigma^V X \iso \bigoplus_{r\geq 0} \tMC \Pow^r_V X.
    \]
\end{cor}

Classes in the image of the inclusion $\tMC \Pow^r_V X \hookrightarrow \bMC \Omega^V \Sigma^V X$ will be said to have {\bf Snaith weight} equal to $r$.

\begin{prop}\label{prop:snaithweights}
    In $\bMC \Omega^\rho S^{\rho+1}$, the classes $e_n$ and $t_n$ have Snaith weight $2^n$.
\end{prop}

\begin{pf}
    We argue by induction.
    The class $t_0$ is the generator of the summand $\tMC \Pow^1_\rho S^1 \iso \tMC S^1$, which establishes the base case. By construction, the Dyer-Lashof operations will double the Snaith weight. The induction step then follows from the formulas $Q_0(t_n) = e_{n+1}$ and $Q_1(t_n) = t_{n+1}$
    given in \cref{hlgyOmRhoSRho1}.
\end{pf}

The following consequence of the Snaith splitting will allow us to work in the Borel setting for our computations.

\begin{proposition}\label{prop:borelcomplete}
    $BP_\R \wedge \OmRhoSRho$ is a Borel-complete $C_2$-spectrum.
\end{proposition}

\begin{proof}
    Letting $\Pow^r_\rho S^1$ denote the $r$-th Snaith summand of $\OmRhoSRho$, we claim that the connectivity of $\Pow^r_\rho S^1$ is strictly increasing in $r$. 
    Since $\Pow^r_\rho S^1$ is finite and we are working $2$-locally, it suffices to establish this claim in $\bMC$-homology. This claim now follows by combining \cref{hlgyOmRhoSRho1}
     with \cref{prop:snaithweights}. Indeed the former implies that $\bH\wedge \OmRhoSRho$ splits as a sum of free $\bH$ modules indexed by monomials in the $t_i's$ and $e_j$'s, and the latter implies that $\bH\wedge \Pow^r_\rho S^1$ 
     corresponds to monomials
     of weight $r$.
     As $e_n$ and $t_n$ have weight $2^n$, one sees 
     by examining the degrees of $e_n$ and $t_n$
     that the connectivity of these summands strictly increases in $r$.

     This implies that, since $BP_\R$ is connective, the connectivity of $BP_\R\wedge \Pow^r_\rho S^1$ is also strictly increasing in $r$. It follows that the map
        \[\bigvee\limits_{r \geq 0} BP_\R\wedge \Pow^r_\rho S^1 \to \prod\limits_{r \geq 0} BP_\R\wedge \Pow^r_\rho S^1\]
    is an equivalence. Since $BP_\R$ is Borel-complete, so is $BP_\R\wedge \Pow^r_\rho S^1$, since $\Pow^r_\rho S^1$ is finite. Further, since a product of Borel-complete $C_2$-spectra is Borel-complete, the $C_2$-spectrum
    \[
        BP_\R \wedge \OmRhoSRho \simeq \bigvee\limits_{r\geq 0} BP_\R\wedge \Pow^r_\rho S^1
    \]
    is Borel-complete.
\end{proof}

\section{The $E_2$-term of the {Borel} Adams spectral sequence}
\label{sec:E2Term}

Our computation in  \cref{prop:A-comod0rhoSrho+1} of the $\cAC$-comodule structure on $\bH_\bigstar\Omega^\rho S^{\rho+1}$ is the input needed to compute the equivariant Adams spectral sequence converging to ${BP_\R}_\bigstar \Omega^\rho S^{\rho+1}$. However, since  ${BP_\R}\wedge \OmRhoSRho$ is Borel complete by  \cref{prop:borelcomplete}, we may instead use the Borel equivariant Adams spectral sequence developed by Greenlees \cite{G}. It is possible to carry out the computation with the genuine equivariant Adams spectral sequence, but the Borel approach is more concise as it allows us to avoid the negative cone computation. The Borel equivariant Adams spectral sequence we will use has signature
\[
E_2=\Ext_{\mathcal{A}^h_\bigstar}(\bH_\bigstar^h,\bH_\bigstar^h(BP_\R\wedge \OmRhoSRho ))\implies \pi_\bigstar(F(E{C_2}_+,BP_\R \wedge \OmRhoSRho))^{\widehat{}}_2
\]
and converges by \cite[Theorem 2.7]{G}. By our  \cref{prop:borelcomplete}, the target is isomorphic to $\pi_\bigstar(BP_\R\wedge \OmRhoSRho)^{\widehat{}}_2$. 
As in \cite[Section 2.2]{LSWX}, we may use a change of rings isomorphism to get a spectral sequence of signature
\begin{equation}\label{eq:ASS}
  E_2=\Ext_{\mathcal{E}^h_\bigstar}(\bH_\bigstar^h,\bH_\bigstar^h \Omega^{\rho}S^{\rho+1} )\implies \pi_\bigstar(BP_\R \wedge \OmRhoSRho)^{\widehat{}}_2,  
\end{equation}
where $(\bH_\bigstar^h,\mathcal{E}^h_\bigstar)$ is the quotient Hopf algebroid of $(\bH_\bigstar^h,\mathcal{A}^h_\bigstar)$ given by
\[
    \mathcal{E}^h_\bigstar=\bH_\bigstar^h[\etau_i]/(\etau_i^2=\asig\etau_{i+1}).
\]
We refer the reader to \cite[Proposition 2.15]{LSWX} for more details on $\mathcal{E}^h_\bigstar$. We will refer to the spectral sequence of \eqref{eq:ASS} as the \emph{Borel Adams spectral sequence of }$BP_\R \wedge \OmRhoSRho$.

\subsection{The $E_1$-page of the $\asig$-Bockstein spectral sequence} 

Following \cite[Section 4.2]{LSWX}, we can compute our Adams $E_2$-page 
\[
    \Ext_{\mathcal{E}^h_\bigstar}(\bH_\bigstar^h,\bH_\bigstar^h\Omega^\rho S^{\rho+1})
\]
via a Ravenel--May spectral sequence obtained from an $\asig$-adic filtration.
More specifically,  the $\asig$-adic filtration
defines a decreasing filtration on the Hopf algebroid $(\bH_\bigstar^h,\mathcal{E}^h_\bigstar)$ in the sense of \cite[Definition A1.3.5]{Ravenel86}.
The $\asig$-adic filtration on $\bH_\bigstar^h\Omega^\rho S^{\rho+1}$ 
similarly defines a decreasing filtration as a $(\bH_\bigstar^h,\mathcal{E}^h_\bigstar)$-comodule, in the sense of \cite[Definition A1.3.7]{Ravenel86}. The resulting spectral sequence, which we will refer to as the \emph{$\asig$-Bockstein spectral sequence}, has signature
\begin{equation}\label{eq:rhoBSS}
E_1-\asig\mathrm{BSS}=\Ext_{(\mathrm{gr}\bH_\bigstar^h,\mathrm{gr}\mathcal{E}^h_\bigstar)}(\mathrm{gr}\bH_\bigstar^h,\mathrm{gr}\bH_\bigstar^h\Omega^\rho S^{\rho+1})\implies 
\Ext_{\mathcal{E}^h_\bigstar}(\bH_\bigstar^h,\bH_\bigstar^h\Omega^\rho S^{\rho+1})
\end{equation}
by \cite[Theorem A1.3.9]{Ravenel86}.

\begin{prop}\label{prop:E1identifications}
    There is an isomorphism of graded rings
    \[
        E_1-\asig\mathrm{BSS}\cong \F_2[\usig^\pm,\asig,v_0,v_1,\ldots,t_0,e_1,e_2,\ldots]/(t_0^2+\usig e_1,r_k | k\ge1)
    \]
    where 
    \[
        r_k=\sum\limits_{j=0}^{k-1}e_{k-j}^{2^j}v_j
    \]
    with Adams bidegrees $|\usig|=(1-\sigma,0)$, $|\asig|=(-\sigma,0)$, $|v_i|=((2^i-1)\rho,1)$, $|t_0|=(1,0)$, and $|e_i|=((2^i-1)\rho,0)$.
\end{prop}

\begin{proof}
The Hopf algebroid structure formulas for $(\bH_\bigstar^h,\mathcal{E}^h_\bigstar)$ given in \cite[Proposition 2.15]{LSWX} imply that 
\[
    (\mathrm{gr}\bH_\bigstar^h,\mathrm{gr}\mathcal{E}^h_\bigstar)\cong\F_2[\usig^\pm,\asig]\otimes(\F_2,E(\etau_1,\etau_2,\dots))
\]
with $\etau_i$'s primitive. 
Since we are filtering with respect to $\asig$,   \cref{prop:borelcomodule} implies that the $\mathrm{gr}\mathcal{E}^h_\bigstar$-comodule 
$\mathrm{gr}\bH_\bigstar^h\Omega^\rho S^{\rho+1}$ is given by
\[
    \F_2[\usig^\pm,\asig,t_0,t_1,\ldots,e_1,e_2,\ldots
    ]/(t_n^2=\usig e_{n+1})
\]
with $u_\sigma^\pm,\asig$, and the $e_n$'s primitive, and
\[
    \psi_R(t_k)=t_k\otimes 1+\sum\limits_{i=1}^{k-1}e_{k-j}^{2^j}\otimes\etau_j.
\] 
To calculate the $E_1$ page of the $\asig$-Bockstein spectral sequence, we closely follow the argument of \cite[Theorem 3.3(b)]{Ravenel93}. We give the $\mathrm{gr}\mathcal{E}^h_\bigstar$-comodule algebra $\mathrm{gr}\bH_\bigstar^h\Omega^\rho S^{\rho+1}$ an increasing filtration by giving each $e_i$ filtration zero and each $t_j$ filtration one. This results in a spectral sequence converging to the group $\Ext_{(\mathrm{gr}\bH_\bigstar^h,\mathrm{gr}\mathcal{E}^h_\bigstar)}(\mathrm{gr}\bH_\bigstar^h,\mathrm{gr}\bH_\bigstar^h\Omega^\rho S^{\rho+1})$ with $E_1$-page given by
\begin{align*}
  &\Ext_{(\mathrm{gr}\bH_\bigstar^h,\mathrm{gr}\mathcal{E}^h_\bigstar)}(\mathrm{gr}\bH_\bigstar^h,\mathrm{gr}\bH_\bigstar^h)
    \otimes_{\bH_\bigstar}\mathrm{gr}\bH_\bigstar^h\Omega^\rho S^{\rho+1}\\
    & \qquad \cong\F_2[\usig^\pm,\asig,v_i,t_j,e_k|i,j\ge0,k\ge 1]/(t_j^2=\usig e_{j+1}).   
\end{align*}
 The coaction formulas $\psi_R(t_j)$ give $d_1$ differentials $d_1(t_j)=r_j$, where $r_0=0$.
    
As in \cite[Theorem 3.3(b)]{Ravenel93}, the $r_i$ form a regular sequence in the $E_1$-page of this filtration spectral sequence,
resulting in an isomorphism
\[
    E_2\cong \F_2[\usig^\pm,\asig,v_0,v_1,\ldots,t_0,e_1,e_2,\ldots]/(t_0^2+\usig e_1,r_i | i\ge1)
\]
for the $E_2$-page of the filtration spectral sequence. Now, each of the generators $\usig$, $\usig^{-1}$, $\asig$, and $v_i$ is a permanent cycle in this filtration spectral sequence since these generators come from $\Ext_{(\mathrm{gr}\bH_\bigstar^h,\mathrm{gr}\mathcal{E}^h_\bigstar)}(\mathrm{gr}\bH_\bigstar^h,\mathrm{gr}\bH_\bigstar^h)$ and therefore admit lifts to $\Ext_{(\mathrm{gr}\bH_\bigstar^h,\mathrm{gr}\mathcal{E}^h_\bigstar)}(\mathrm{gr}\bH_\bigstar^h,\mathrm{gr}\bH_\bigstar^h\Omega^\rho S^{\rho+1})$. 
The generators $t_0$ and $e_i$ are permanent cycles because they are represented by $\mathrm{gr}\mathcal{E}^h_\bigstar$-comodule primitives in $\mathrm{gr}\bH_\bigstar^h\Omega^\rho S^{\rho+1}$ and therefore admit lifts to $\Ext^0_{(\mathrm{gr}\bH_\bigstar^h,\mathrm{gr}\mathcal{E}^h_\bigstar)}(\mathrm{gr}\bH_\bigstar^h,\mathrm{gr}\bH_\bigstar^h\Omega^\rho S^{\rho+1})$. It follows that the filtration spectral sequence collapses on $E_2$. 

Finally, the relations $t_0^2+\usig e_1=0$ and $r_i=0$ hold in the ring 
\[\Ext_{(\mathrm{gr}\bH_\bigstar^h,\mathrm{gr}\mathcal{E}^h_\bigstar)}(\mathrm{gr}\bH_\bigstar^h,\mathrm{gr}\bH_\bigstar^h\Omega^\rho S^{\rho+1})\]
by way of the relation $t_0^2+\usig e_1=0$ in $\mathrm{gr}\bH_\bigstar^h\Omega^\rho S^{\rho+1}$ and the coaction formulas $\psi_R(t_i)$, respectively. Indeed, the former relation gives a relation between primitives in $\Ext^0$ and the latter coaction formulas define differentials in the cobar complex computing $\Ext_{(\mathrm{gr}\bH_\bigstar^h,\mathrm{gr}\mathcal{E}^h_\bigstar)}(\mathrm{gr}\bH_\bigstar^h,\mathrm{gr}\bH_\bigstar^h\Omega^\rho S^{\rho+1})$.
Therefore, there are no nontrivial extension problems in the filtration spectral sequence, and the 
proposition follows.
\end{proof}

\begin{rmk}
\label{rmk:E(n)nonregular}
    As discussed in the proof above, the elements $r_i$ form a regular sequence, which is a crucial fact in the calculation of the corresponding Ext group. The regularity breaks down, however, if one works over a subalgebra $\cEC(n)$ rather than $\cEC$, leading to a more complicated answer.
\end{rmk}

\subsection{The Borel Adams spectral sequence for $BP_\R\wedge \OmRhoSRho$}  \cref{prop:E1identifications} gives us an explicit description of the $E_1$-page of the $\asig$-Bockstein SS of \eqref{eq:rhoBSS}. We will show that, up to use of the Leibniz rule, the only classes that support nontrivial differentials in the 
$\asig$-Bockstein SS are the powers of $\usig$. There is therefore at most one nontrivial differential on each page, up to use of the Leibniz rule, which will allow us to use the following lemma to compute the homology with respect to these differentials.

\begin{lem}\label{lem:kernelgenerators}
    Let $R$ be a cdga over $\F_2$ such that a set $A=\{x_\alpha\}\subset R$ generates $R$ as an algebra, and suppose that all but one $x_\alpha$ is a cycle, i.e.
    \[
    d(x_\alpha)=\begin{cases}y\neq0&\alpha=1\\
    0&else
    \end{cases}
    \]
    If $B=\{a_\beta\}\subset Z(R)$ is a set of generators of the annihilator ideal of $y$, then 
\[(A\setminus\{x_1\})\cup\{x_1^2\}\cup\{a_\beta,a_\beta\cdot x_1\}_{\beta\in B}\]
    is a set of algebra generators for $H_*(R)$.
\end{lem}
\begin{proof}
    It suffices to prove the claim for $\ker(d)$ since there is a surjection of algebras $\ker(d)\to H_*R$. We claim first that it suffices to show that if $p\in\ker(d)$, then $p$ is in the subalgebra generated by \[(A\setminus\{x_1\})\cup\{x_1^2\}\cup x_1\cdot\mathrm{ann}(y)\]
    Indeed, any element in $\mathrm{ann}(y)$ can be written as a sum of monomials of the form $r(x_\alpha)a_\beta$, where $r(x_\alpha)$ is a monomial in the $x_\alpha$'s. If the exponent of $x_1$ in $r(x_\alpha)$ is even, then $r(x_\alpha)a_\beta\cdot x_1$ is in the subalgebra generated by
\[(A\setminus\{x_1\})\cup\{x_1^2\}\cup\{a_\beta\cdot x_1\}_{\beta\in B}\]
    If the exponent of $x_1$ in $r(x_\alpha)$ is odd, then $r(x_\alpha)a_\beta\cdot x_1$ is in the subalgebra generated by \[(A\setminus\{x_1\})\cup\{x_1^2\}\cup\{a_\beta\}_{\beta\in B}\]
    
    Since $A$ generates $R$ as an algebra, every element in $R$ may be written as a sum 
    \[p=m_1+\cdots+m_n\]
    of monomials $m_i$ in the $x_\alpha$'s, and we will prove by induction on $n$ that if $p\in\ker(d)$, then $p$ is in the subalgebra of $R$ generated by\[(A\setminus\{x_1\})\cup\{x_1^2\}\cup x_1\cdot\mathrm{ann}(y)\]

    Let $n=1$, so that $p$ is a monomial. If $p$ is a monomial in the set $A\setminus\{x_1\}$, then we are done. Otherwise $p=x_1^Np'$ where $p'$ is a monomial in $A\setminus\{x_1\}$. If $N$ is even, then we are done. Otherwise
    \[0=d(p)=d(x_1\cdot x_1^{N-1}p')=y\cdot x_1^{N-1}p'\]
    so that $x_1^{N-1}p'\in\mathrm{ann}(y)$, finishing the proof in this case.

    If $n>1$, then if $d(m_i)=0$ for any $i$, the equation $d(p)=0$ implies that
    \[d(m_1+\cdots+m_{i-1}+m_{i+1}+\cdots m_n)=0\]
    and we may conclude by induction. Otherwise, $d(m_i)\neq0$ for all $i$, and  we may write $m_i=x_1^{k_i}q_i$ where $k_i>0$ is odd and $d(q_i)=0$. This implies that 
    \[0=d(p)=y\sum\limits_{i=1}^nx_1^{k_i-1}q_i\]
    so that 
    $\sum\limits_{i=1}^nx_1^{k_i-1}q_i\in\mathrm{ann}(y)$, completing the proof.
\end{proof}

\begin{rmk}\label{rmk:laurentgenerators}
    We will apply the preceding lemma to the case in which the generator $x_1$ is invertible. That is, let $R$ be as in the Lemma, and consider the cdga $S=R[x_1^{-1}]$. A priori, to generate $S$ as an algebra, one must include $x_1^{-1}$ into $\{x_\alpha\}$, and then since $x_1^{-1}$ is not a cycle, the assumption that all but one $x_\alpha$ is a cycle does not hold. However, over $\F_2$, the class $x_1^{-2}$ is a cycle since it is a square in $S$, so we may instead add this to the list $\{x_\alpha\}$. The key assumption that $x_1$ is the only non-cycle then still holds, and $\{x_\alpha\}$ generates $S$
 as an algebra because of the relation $x_1^{-1}=x_1^{-2}\cdot x_1$.
 \end{rmk}

To determine the differentials in the $\asig$-BSS, we will make use of the map of $\asig$-BSS's induced by the unit map $S^0\to \OmRhoSRho$. The $\asig$-BSS for the sphere 
(in the sense of \cref{eq:rhoBSS})
was shown by Li--Shi--Wang--Xu to be isomorphic to the associated graded homotopy fixed point spectral sequence for $BP_\R$, and the latter was computed by Greenlees--Meier \cite{GM},
following the earlier $\R$-motivic computation of Hill \cite[Theorem~3.1]{Hill}.

\begin{proposition}[{\cite[Theorem~4.7]{LSWX}\cite[Proposition A.2]{GM}}]\label{prop:annihilatorBPR}
    In the $\asig$-Bockstein SS converging to $\Ext_{\mathcal{E}^h_\bigstar}(\bH_\bigstar^h,\bH_\bigstar^h)$, the $E_{2^{n+1}-1}$-page is given by the subalgebra of
\[\F_2[\usig^\pm,\asig,v_0,v_1,\ldots]/(\asig^3v_1,\ldots,\asig^{2^n-1}v_{n-1})\]
generated by $\usig^{2^n}$, $\asig$, the $v_i$'s, and $\usig^{2^{j+1}k}v_j$ for $j<n-1$ and $k\in\Z$. There is a nonzero differential
    \[d_{2^{n+1}-1}(\usig^{2^{n}})=\asig^{2^{n+1}-1}v_n.\]
\end{proposition}

Taking $n=\infty$, Greenlees--Meier show there are no nontrivial extensions in the HFPSS for $BP_\R$ and deduce a presentation of $\pi_\bigstar BP_\R$ (\cite[Proposition A.4]{GM}). Combining again with \cite[Theorem~4.7]{LSWX}, one has the following.

\begin{thm}\label{thm:hilldifferentials}
    The algebra $\Ext_{\mathcal{E}^h_\bigstar}(\bH_\bigstar^h,\bH_\bigstar^h)$ is generated over $\F_2[\asig]$ by classes $\usig^{2^{n+1}k}v_n$, for $n\geq 0$ and $k\in\Z$, subject to relations
    \begin{enumerate}
        \item ($\asig$-torsion): $\asig^{2^{n+1}-1}\cdot \usig^{2^{n+1}k}v_n = 0$
        \item (product of generators): 
        \[
            \usig^{2^{n+1}k}v_n \cdot \usig^{2^{i+1}\ell}v_i 
            = v_n \cdot \usig^{2^{i+1} (2^{n-i}k+\ell)} v_i
        \]
        for $n \geq i$.
    \end{enumerate}
\end{thm}

\begin{cor}\label{cor:permanentcycles}
    In the $\asig$-BSS of \eqref{eq:rhoBSS}, the classes $\asig$, $t_0$, $e_j$, and $\usig^{2^{n+1}k}v_n$  are permanent cycles, for all $j\geq 1$, $n\ge0$, and $k\in\Z$. 
\end{cor}

\begin{proof}
The classes $\asig$, $t_0$, and $e_j$ are all $\mathcal{E}^h_\bigstar$ comodule primitives in $\bH^h_\bigstar\Omega^\rho S^{\rho+1}$, hence the classes on $E_1$ admit lifts to $\Ext^0$ and thus are permanent cycles. The classes $\usig^{2^{n+1}k}v_n$ for $n\ge0$ and $k\in\Z$ are permanent cycles by  \cref{thm:hilldifferentials}.
\end{proof}

Beginning with our $E_1$-page for the $\asig$-BSS for $BP_\R\wedge \OmRhoSRho$
\[E_1\cong \F_2[\usig^\pm,\asig,v_0,v_1,\ldots,t_0,e_1,e_2,\ldots]/(t_0^2+\usig e_1,r_i | i\ge1),\]
we will show that, up to use of the Leibniz rule, the only nonzero differentials are
   \[d_{2^{n+1}-1}(\usig^{2^{n}})=\asig^{2^{n+1}-1}v_n\]
from  \cref{prop:annihilatorBPR}. Following  \cref{lem:kernelgenerators}, this will require computing the annihilator ideal of $\asig^{2^{n+1}-1}v_n$ on $E_{2^{n+1}-1}$. For this, we will need an algebraic lemma showing that various rings related to $E_{1}$ are integral domains. 

\begin{restatable}{lemma}{LemmaNineNine}
\label{lemma:integraldomain}
For all $k\ge 0$, the ring
    \[
        C_k=\frac{\F_2[\usig^{\pm 2^{k}},\asig,v_{k},\ldots,t_0,e_1,e_2,\ldots]}{(t_0^{2^{k+1}}+\usig^{2^{k}}e_1^{2^{k}},\overline{r_{k+2}},\ldots)}
    \]
is an integral domain, where the relations $\overline{r_{j}}$ for $j\ge k+2$ are the projections 
modulo $(v_0,\ldots,v_{k-1})$
of the relations
\[
r_j=\sum\limits_{i=0}^{j-1}e_{j-i}^{2^i}v_i
\]
from \cref{prop:E1identifications}.
\end{restatable}

The proof
is a modification of the proof of \cite[Lemma 4.15(b)]{RavenelWilson77}. 
For the purpose of being self-contained, we provide the details of our modification, though we defer the proof to \cref{ProofLemma99} in order to not disrupt the narrative. 

\begin{proposition}
For all $n\ge 0$, the following hold for the $E_{2^{n+1}-1}$-page of the $a_\sigma$-BSS of \eqref{eq:rhoBSS}.
\begin{enumerate}
\item $E_{2^{n+1}-1}$ is isomorphic to the subalgebra of 
\[
\frac{\F_2[\usig^\pm,\asig,v_0,v_1,\ldots,t_0,e_1,e_2,\ldots]}{(t_0^2+\usig e_1,r_1,r_2,\ldots,\asig^{2^{j+1}-1}v_j|0\le j\le n-1)}
\]
generated by
\[
\{\usig^{\pm 2^{n}},\asig,v_n,v_{n+1},\ldots,t_0,e_1,e_2,\ldots,\usig^{2^{j+1}k}v_j|0\le j\le n-1,k\in\Z\}
\] 
\item The annihilator ideal of $\asig^{2^{n+1}-1}$ in $E_{2^{n+1}-1}$ is 
\[
    \mathrm{ann}(\asig^{2^{n+1}-1})=(\usig^{2^{j+1}k}v_j|0\le j\le n-1,k\in\Z)
\]
\item 
Passing to a quotient by the annihilator ideal gives

\[
    E_{2^{n+1}-1}/\mathrm{ann}(\asig^{2^{n+1}-1})\cong \frac{\F_2[\usig^{\pm 2^{n}},\asig,v_n,v_{n+1},\ldots,t_0,e_1,e_2,\ldots]}{(t_0^{2^{n+1}}+\usig^{2^{n}}e_1^{2^{n}},\overline{r_{n+1}},\overline{r_{n+2}},\ldots)}
\]
where the relations $\overline{r_{j}}$ for $j\ge n+1$ are the projections of the relations $r_j$ from \cref{prop:E1identifications} mod $(v_0,\ldots,v_{n-1})$.
\item The annihilator ideal of $\asig^{2^{n+1}-1} v_n$ in $E_{2^{n+1}-1}$ is \[(e_1^{2^n},\usig^{2^{j+1}k}v_j|0\le j\le n-1,k\in\Z)\]
\end{enumerate}
\end{proposition}

\begin{pf}  
We begin with the base case $n=0$. The description of $E_1$ in item (1) is \cref{prop:E1identifications}. Items (2) and (3) follow in {the $n=0$ case} from the fact that $E_1$ is $\asig$-torsion free. For item (4), since $r_1=v_0e_1$, the claim follows from the isomorphism $E_1\cong C_0/(v_0e_1)$, where $C_0$ is the integral domain of  \cref{lemma:integraldomain} with $k=0$.

For the inductive step, assume that items (1)-(4) hold for $E_{2^{n}-1}$. By item (1), all the algebra generators of $E_{2^n-1}$ are permanent cycles except $\usig^{2^{n}}$ (using the reasoning in \cref{rmk:laurentgenerators}), so the only nonzero differential (up to use of the Leibniz rule) on $E_{2^{n}-1}$ is
\[d_{2^{n}-1}(\usig^{2^{n}})=\asig^{2^{n}-1}v_{n-1}.\]
Using \cref{lem:kernelgenerators}  and item (4) for $E_{2^n-1}$, the desired description in item (1) for $E_{2^{n+1}-1}$ follows. Item (2) for $E_{2^{n+1}-1}$ now follows from the fact that the annihilator of $\asig^{2^{n+1}-1}$ in the larger algebra 
\[
\frac{\F_2[\usig^\pm,\asig,v_0,v_1,\ldots,t_0,e_1,e_2,\ldots]}{(t_0^2+\usig e_1,r_1,r_2,\ldots,\asig^{2^{j+1}-1}v_j|0\le j\le n-1)}
\]
is $(v_0,\ldots,v_{n-1})$. Indeed this implies that the annihilator of $\asig^{2^{n+1}-1}$ in $E_{2^{n+1}-1}$ is $E_{2^{n+1}-1}\cap (v_0,\ldots,v_{n-1})$, from which the claim follows.

Item (3) for $E_{2^{n+1}-1}$ then follows directly from items (1) and (2) for $E_{2^{n+1}-1}$. To establish item (4) for $E_{2^{n+1}-1}$, one must compute the annihilator of $v_n$ in $E_{2^{n+1}-1}/\mathrm{ann}(\asig^{2^{n+1}-1})$, which is given  by item (3) for $E_{2^{n+1}-1}$. It suffices now to observe that there is a ring isomorphism
\[C_{n}/(v_ne_1^{2^n})\cong E_{2^{n+1}-1}/\mathrm{ann}(\asig^{2^{n+1}-1}),\]
where $C_n$ is the integral domain of \cref{lemma:integraldomain}.
\end{pf}

Setting $n=\infty$, we have the following.

\begin{cor}
\label{cor:EinftyBSS}
    The $E_\infty$-page of the $\asig$-BSS for $BP_\R\wedge \OmRhoSRho$ is isomorphic to the subalgebra of 
\[
\frac{\F_2[\usig^\pm,\asig,v_0,v_1,\ldots,t_0,e_1,e_2,\ldots]}{(t_0^2+\usig e_1,r_1,r_2,\ldots,\asig^{2^{j+1}-1}v_j|j\ge0)}
\]
generated by
\[
\{\usig^{2^{j+1}k}v_j,\asig,t_0,e_1,e_2,\ldots|j\ge0,k\in\Z\}.
\]  
\end{cor}

Our analysis of the Adams spectral sequence will use the following additional properties of the $E_\infty$-page.

\begin{lem}\label{lem:rhoBSSevenness}
    The $E_\infty$-page of the $\asig$-BSS for $BP_\R\wedge \OmRhoSRho$ satisfies: 
    \begin{enumerate}
        \item 
        The $E_\infty$-page vanishes in degrees of the form $j\rho-1$ for $j\in\Z$
        \item In any fixed degree,
        the sequence of ideals $\{(\asig^n)\}_{n\ge0}$ in $E_\infty$ is constant for $n$ sufficiently large.
    \end{enumerate}
\end{lem}

It is not immediately clear from 
\cref{cor:EinftyBSS}
that \cref{lem:rhoBSSevenness} holds, since the stated properties
do not hold in the larger algebra
\begin{equation}
\label{LargerAlgebra}    
\frac{\F_2[\usig^\pm,\asig,v_0,v_1,\ldots,t_0,e_1,e_2,\ldots]}{(t_0^2+\usig e_1,r_1,r_2,\ldots,\asig^{2^{j+1}-1}v_j|j\ge0)}.
\end{equation}
For example, $\frac{\asig}{\usig}$ is an element of degree -1 in the larger algebra. Similarly, the infinite sequence of elements $\frac{\asig^2 }{\usig}e_1$, $\big(\frac{\asig^2 }{\usig}\big)^3e_2$, $\big(\frac{\asig^{2}}{\usig}\big)^7 e_3$, \dots are all in degree 0 in the larger algebra.
However, we give the following straightforward degree arguments.

\begin{proof}
We begin with the proof of item (1). Let $m$ be a monomial in the generators 
\[
    \{\usig^{2^{j+1}k}v_j,\asig,t_0,e_1,e_2,\ldots|j\ge0,k\in\Z\}.
\]  
Using relation (2) in \cref{thm:hilldifferentials}, we may express $m$ in the form
\[
m= r \cdot
\asig^\alpha t_0^\beta
\usig^{2^{n+1}k}v_n
\]
for 
$r$ a monomial in the $v_i$'s and $e_i$'s and for
some $\alpha,\beta,n\ge0$, $k\in\Z$.

We assume now that $m$ lives in degree $j\rho-1$. Writing the degree of an arbitrary element as $c+w\sigma$, we are assuming that $c-w$ is equal to $-1$. Note that the monomial $r$ satisfies $c-w=0$, as the same is true of each $v_i$ and $e_i$. On the other hand, $\asig$ and $t_0$ both satisfy $c-w=1$, while $\usig^{2^{n+1}k}v_n$ satisfies $c-w=2k\cdot 2^{n+1} = k\cdot 2^{n+2}$.
Therefore, we learn that
\[
    \alpha + \beta + k\cdot 2^{n+2}=-1.
\]
Since $\alpha,\beta\ge0$, it follows then that $k<0$ and hence $\alpha+\beta\ge 2^{n+2}-1$. The lemma will follow then from the following claim: 
if $\alpha+\beta\ge 2^{n+2}-2$, then $\asig^\alpha t_0^\beta v_n$ vanishes
in the algebra \eqref{LargerAlgebra}.

We prove the claim by induction on $n$, and the $n=0$ case follows from the relations $\asig v_0=0$ and $t_0^2v_0=\usig e_1v_0=0$. 
Assume then that for all $i<n$, $\asig^\alpha t_0^\beta v_i$ vanishes whenever $\alpha+\beta\ge 2^{i+2}-2$. 
Let $\alpha+\beta\ge 2^{n+2}-2$.
Then either $\alpha \geq 2^{n+1}-1$ or else $\beta \geq 2^{n+1}$. In the former case, the relation $\asig^{2^{n+1}-1}v_n=0$ implies that $\asig^\alpha t_0^\beta v_n$ vanishes.
Write $\alpha+\beta=2^{n+2}-2+l$ for $l\ge 0$ and $\beta=2^{n+1}+r$ for $r\ge0$. Then
\begin{align*}
\asig^\alpha t_0^\beta v_n&=\asig^{2^{n+1}-2+l-r}t_0^{2^{n+1}+r}v_n\\
&=\usig^{2^n}\asig^{2^{n+1}-2+l-r}t_0^{r}e_1^{2^n}v_n\\
&=\usig^{2^n}\asig^{2^{n+1}-2+l-r}t_0^{r}(e_{n+1}v_0+\cdots+e_2^{2^{n-1}}v_{n-1})
\end{align*}
Now since
\[(2^{n+1}-2+l-r)+r\ge 2^{n+1}-2\]
the right hand side is zero by the inductive hypothesis.

Item (2) follows by similar degree arguments. 
In degrees $j\rho+n$, with $j\in\Z$ and $n\geq 0$,
the statement follows from item (1) since any class 
divisible by $\asig^{n+1}$ will be the product of $\asig^{n+1}$ with a class in degree $(j+n+1)\rho-1$.
Suppose then that $w>0$, consider the bidegree 
$j\rho+w\sigma$ for $j\in\Z$,
and suppose that the monomial
\[
    m= r \cdot
\asig^n t_0^\beta 
\usig^{2^{\ell+1}k}v_\ell
\]
is in degree 
$j\rho+w\sigma$
for 
$n$, $\beta$, $\ell\geq 0$, 
$k\in\Z$, and $r$ a monomial in the generators $v_i$ and $e_j$. 
This implies the equation
\[
    w = -n-\beta-2\cdot 2^{\ell+1}k
\]
or in other words
\[
    n+\beta+w=-2^{\ell+2}\cdot k,
\]
which implies $k<0$. 
By the proof of item (1), we must have $n+\beta<2^{\ell+2}-2$, which implies that $w>2^{\ell+2}(-k-1)+2$. If $k\neq -1$, then $w>2^{\ell+2}+2$, so since $w$ is fixed, there are only finitely many possibilities for $\ell$. Again since $n+\beta<2^{\ell+2}-2$, this means there are only finitely many possibilities for $n$, as desired.

It remains to consider the case above when $k=-1$. This gives the equation $n+\beta+w=2^{\ell+2}$ and since $n<2^{\ell+1}-1$, we may write $n=2^{\ell+1}-1-h$ for $h>0$. 
Evaluating the $\sigma$ degree of the above monomial $m$ gives the equation 
\[
    j+w = N -(2^{\ell+1}-1-h)+(2^{\ell+1}+2^{\ell}-1) 
    = N + h + 2^\ell,
\]
where $N\geq 0$ is the $\sigma$-degree of $r$.
Since $j$ and $w$ are fixed, 
there are only finitely many possibilities for 
$N$, $h$, and $\ell$.
It follows that there are only finitely many possibilities for $n$.
\end{proof}

\subsection{Hidden extensions in the $\asig$-BSS}\label{subsec:hiddenextensions}

{Our \cref{cor:EinftyBSS} computes the $E_\infty$-page of the $\asig$-BSS for $BP_\R\wedge\OmRhoSRho$, and therefore determines the $E_2$-page of the Borel Adams SS for $BP_\R\wedge\OmRhoSRho$ as an $\F_2$-vector space. However, we do not attempt to resolve all the extension problems in  the $\asig$-BSS.
For now, we demonstrate how some of these extensions may be resolved using the transchromatic Massey products in 
{$\Ext_{\cE_\bigstar}(\bMC,\bMC)$}
observed by Beaudry--Hill--Shi--Zeng \cite{BHSZtranschromatic}.

Beaudry--Hill--Shi--Zeng studied a family of Massey products in the $E_2$-page of the $\R$-motivic Adams spectral sequence of algebraic cobordism $MGL$. Via Betti realization this determines a corresponding family of Massey products in the $C_2$-equivariant Adams SS for $MU_\R$. Mapping further to the Borel Adams SS of $BP_\R\wedge\OmRhoSRho$, this determines a family of Massey products on the $E_2$-page of (\ref{eq:ASS}). Beaudry--Hill--Shi--Zeng discuss more general families of Massey products than what appears below; we highlight a special case just to illustrate some of the extensions appearing here. The following is obtained by taking $k=l=2^n$, $r=t=0$, and $s=2^{n+1}-1$ in \cite[Theorem 2.10]{BHSZtranschromatic}.

\begin{prop}
    For all $n\ge0$, there is a containment
    \[\asig^{2^{n+1}}v_{n+1}\in\langle v_n,\asig^{2^{n+1}-1},v_n\rangle\subset\Ext_{\mathcal{E}^h_\bigstar}(\bH_\bigstar^h,\bH_\bigstar^h\Omega^\rho S^{\rho+1}).
\]
\end{prop}
\begin{proof}
    The maps of Hopf algebroids
    \[(\bH^\R_\bigstar,\mathcal{E}^\R_\bigstar)\to (\bH_\bigstar,\mathcal{E}_\bigstar)\to (\bH^h_\bigstar,\mathcal{E}^h_\bigstar)\to (\bH_\bigstar^h\Omega^\rho S^{\rho+1},\bH_\bigstar^h\Omega^\rho S^{\rho+1}\otimes_{\bH^h_\bigstar}\mathcal{E}^h_\bigstar)\]
    induce maps of dga's between the cobar complexes computing the corresponding Ext groups. The result of Beaudry--Hill--Shi--Zeng gives the corresponding containment in Ext over $(\bH^\R_\bigstar,\mathcal{E}^\R_\bigstar)$, so the claim follows from naturality of Massey products under maps of dga's. 
\end{proof}

This proposition allows us to resolve extension problems in the $\asig$-BSS for $BP_\R\wedge\OmRhoSRho$ using shuffling formulas for Massey products. In fact, using the relations $r_n$ of \cref{prop:E1identifications}, which hold in 
\[
\Ext_{\mathcal{E}^h_\bigstar}(\bH_\bigstar^h,\bH_\bigstar^h\Omega^\rho S^{\rho+1})
\]
via the coproduct formula on $t_n$ in \cref{prop:borelcomodule},
a straightforward induction argument shows that
\begin{equation}
\label{eq:vn_kills_e1power}
e_1^{2^{n+1}-1}v_n=0\in    \Ext_{\mathcal{E}^h_\bigstar}(\bH_\bigstar^h,\bH_\bigstar^h\Omega^\rho S^{\rho+1}).
\end{equation}
This implies that the Massey product $\langle \asig^{2^{n+1}-1},v_n,e_1^{2^{n+1}-1}\rangle\subset \Ext_{\mathcal{E}^h_\bigstar}(\bH_\bigstar^h,\bH_\bigstar^h\Omega^\rho S^{\rho+1})$ is defined, and this yields the following hidden multiplications.

\begin{prop}\label{prop:hiddenextensions}
    For any element
    \[x\in \langle \asig^{2^{n+1}-1},v_n,e_1^{2^{n+1}-1}\rangle\subset \Ext_{\mathcal{E}^h_\bigstar}(\bH_\bigstar^h,\bH_\bigstar^h\Omega^\rho S^{\rho+1}),\]
    one has the relation
    \[v_n\cdot x = \asig^{2^{n+1}} v_{n+1}e_1^{2^{n+1}-1}\in \Ext_{\mathcal{E}^h_\bigstar}(\bH_\bigstar^h,\bH_\bigstar^h\Omega^\rho S^{\rho+1}).\]
\end{prop}

Note that the elements $v_i$ come from $\Ext_{\mathcal{E}^h_\bigstar}(\bH_\bigstar^h,\bH_\bigstar^h)$, and the elements $e_i$ are well-defined elements of $\Ext^0$, as they are comodule primitives by \cref{prop:borelcomodule}.

\begin{proof}
One has the shuffling formula
\[
    v_n x\in v_n\langle \asig^{2^{n+1}-1},v_n,e_1^{2^{n+1}-1}\rangle=
    \langle v_n,\asig^{2^{n+1}-1},v_n\rangle e_1^{2^{n+1}-1},
\]
{and the right-hand product contains $\asig^{2^{n+1}} v_{n+1}e_1^{2^{n+1}-1}$.}
It therefore remains to show that the set $\langle v_n,\asig^{2^{n+1}-1},v_n\rangle e_1^{2^{n+1}-1}$ is a singleton. The indeterminacy of the Massey product $\langle v_n,\asig^{2^{n+1}-1},v_n\rangle$ 
{consists of $v_n$-multiples. But these are}
killed by $e_1^{2^{n+1}-1}$ via the relation 
\eqref{eq:vn_kills_e1power}, completing the proof. 
\end{proof}

The following hidden extension is displayed in our weight 2 chart in \cref{fig:Snaith0and2}. 

\begin{example}
        In the ring
    \[\Ext_{\mathcal{E}^h_\bigstar}(\bH_\bigstar^h,\bH_\bigstar^h\Omega^\rho S^{\rho+1})\]
    {we will deduce} the equation
    \[v_0\cdot t_0^2=\asig^2 v_1e_1.\]
    Indeed this follows from the above proposition along with the fact that 
    \[t_0^2\in\langle \asig,v_0,e_1\rangle.\]
    This fact, in turn, follows from the fact that in the cobar complex computing   
   \[
   \Ext_{\mathcal{E}^h_\bigstar}(\bH_\bigstar^h,\bH_\bigstar^h\Omega^\rho S^{\rho+1})
   \]
   one has the differentials
   \begin{align*}
       d(\usig)&=\asig v_0\\
       d(t_1)&=v_0 e_1,
   \end{align*}
   which follow from the formula for $\eta_R(u_\sigma)$
   and the coaction on $t_1$ from \cref{prop:borelcomodule}, respectively.
   {Recall from \cref{hlgyOmRhoSRho1} that in $\bH_\bigstar^h \Omega^\rho S^{\rho+1}$, we have the relation $t_0^2 = \usig e_1 + \asig t_1$.}
\end{example}

\begin{example}
    In the ring
    \[\Ext_{\mathcal{E}^h_\bigstar}(\bH_\bigstar^h,\bH_\bigstar^h\Omega^\rho S^{\rho+1})\]
    one has the relation 
        \[v_1\cdot (e_1t_0^4+\asig^2 e_2t_0^2)=\asig ^4v_2 e_1^3,\]
        which comes from a containment
    \[e_1t_0^4+\asig^2 e_2t_0^2\in\langle \asig^3,v_1,e_1^3\rangle.\]
    The proof of this containment proceeds exactly as in the previous example, and we leave the details to the interested reader.
\end{example}

We expect that one may identify explicit elements in the brackets 
\[\langle \asig^{2^{n+1}-1},v_n,e_1^{2^{n+1}-1}\rangle\]
for all $n$ as above, which yield more hidden extensions by \cref{prop:hiddenextensions}. However, since we will not systematically solve all extension problems, we do not pursue this further. 

\section{The {Borel} Adams spectral sequence}

The $E_\infty$-page of the $\asig$-Bockstein spectral sequence for $BP_\R\wedge\OmRhoSRho$ was described in \cref{cor:EinftyBSS}. 
This gives the associated graded, with respect to the $\asig$-filtration, of $\Ext_{\mathcal{E}^h_\bigstar}(\bH_\bigstar^h,\bH_\bigstar^h\Omega^\rho S^{\rho+1})$, which is the $E_2$-term for the Borel Adams spectral sequence for $BP_\R\wedge\OmRhoSRho$. We now show that there are no nonzero Adams differentials.

\begin{cor}\label{cor:adamscollapse}
    The Borel Adams SS for $BP_\R\wedge \OmRhoSRho$ collapses on the $E_2$-page.
\end{cor}

\begin{proof}
    The class $t_0$ is the fundamental class for the Snaith summand $\Pow^1_\rho S^1\simeq S^1$, and so it is a permanent cycle in the Borel Adams SS for $BP_\R\wedge \OmRhoSRho$.
    The classes $\asig$ and $\usig^{2^{n+1}k}v_n$ are all permanent cycles since they come from $BP_\R$, as the Borel Adams SS for $BP_\R$ collapses on the $E_2$-page. By item (1) of \cref{lem:rhoBSSevenness}, 
    the $\asig$-adic associated graded of $E_2$ vanishes in stems of the form $j\rho-1$, hence 
    the $E_2$ page also vanishes and thus the $e_i$'s are permanent cycles in the Borel Adams SS for degree reasons. These classes generate the $\asig$-adic associated graded of the $E_2$-page as an algebra by \cref{cor:EinftyBSS}.
    
Giving $\asig$ filtration 1 and all other generators filtration $0$ in the ring 
\[\F_2[\usig^{2^{j+1}k}v_j,\asig,t_0,e_1,e_2,\ldots|j\ge0,k\in\Z],
\]
we therefore have a map of filtered rings
\[\F_2[\usig^{2^{j+1}k}v_j,\asig,t_0,e_1,e_2,\ldots|j\ge0,k\in\Z]\to E_2\]
with the property that it is surjective on associated graded. A map of filtered abelian groups that induces a surjection on associated graded is not automatically a surjection, but a diagram chase shows that this is true when the target has only finitely many nonzero filtration quotients. By item (2) of \cref{lem:rhoBSSevenness}, $E_2$ has the property that in any fixed bidegree there are only finitely many nonzero filtration quotients, hence we see that the above map is surjective. Surjectivity implies that $E_2$ is generated as an algebra by permanent cycles, completing the proof.
\end{proof}

As noted previously in \cref{rmk:hiddenextns}, we have considered some, but not all, of the hidden extensions in the Adams spectral sequence.

\subsection{Consequences for $BP_\R\wedge\OmRhoSRho$}
We may group the Snaith splitting into an even-odd decomposition
\[
    \Sigma^\infty_{C_2} \OmRhoSRho \simeq L_{ev}\vee L_{odd}
\]
where $L_{ev}$ is the wedge of all the even weight summands and $L_{odd}$ is the wedge of all the odd weight summands. The following proposition tells us we only need to determine ${BP_\R}_\bigstar L_{ev}$ to determine ${BP_\R}_\bigstar \Omega^{\rho}S^{\rho+1}$.

\begin{proposition}\label{prop:evenoddsplit}
    The map
    \[t_0:\Sigma BP_\R\wedge L_{ev}\to BP_\R\wedge L_{odd}\]
    is an equivalence of $BP_\R$-modules.
\end{proposition}

\begin{proof}
    This is a map between Borel-complete $C_2$-spectra that induces an equivalence on the underlying spectra by \cite[Theorem C]{Ravenel93}.
\end{proof}

\begin{proposition}\label{prop:resinjective}
    The restriction map
    \[\mathrm{res}:(BP_\R)_{*\rho}\Omega^\rho S^{\rho+1}\to BP_{2*}\Omega^2S^3\]
    is an isomorphism. Moreover, the $C_2$-spectrum $BP_\R\wedge L_{ev}$ is strongly even.
\end{proposition}

Recall that a $C_2$-spectrum $X$ is {\bf strongly even} if, for all $n\in\Z$, the homotopy Mackey functors $\underline{\pi}_{n\rho}X$ are  constant (or equivalently have bijective restriction maps) and the Mackey functors $\underline{\pi}_{n\rho-1}X$ vanish.

\begin{proof}
    Fix $n\in \Z$.
    We first show that the restriction map is injective. By the relation $\mathrm{image}(\asig)=\ker(\mathrm{res})$, since $|\asig|=-\sigma$, it suffices to show that 
    \[(BP_\R)_{n\rho+\sigma}\Omega^\rho S^{\rho+1}=0,\]
    which follows directly from item (1) of \cref{lem:rhoBSSevenness}. 
    To see that $\mathrm{res}$ is surjective, note that $\mathrm{res}$ is a ring map, and so it suffices to show that the $BP_*$-algebra generators $y_n$ of $BP_{2*}\Omega^2S^3$ (see \cref{RavenelComputation}) lift to $(BP_\R)_{*\rho}\Omega^\rho S^{\rho+1}$. These classes are detected in $H\F_2$-homology, and according to the proof of \cite[Theorem~4.1]{BW}, the homology classes $y_n$ lift to equivariant homology classes $e_n$.
    By \cref{cor:permanentcycles} and \cref{cor:adamscollapse}, these $\bH$-homology classes lift to $BP_\R$-homology classes.
 
To see that $BP_\R\wedge L_{ev}$ is strongly even, note that 
\[
    (BP_\R)_{n\rho-1}\Omega^\rho S^{\rho+1}=(BP_\R)_{(n-1)\rho+\sigma}\Omega^\rho S^{\rho+1}=0,
\]
so $(BP_\R)_{n\rho-1}L_{ev}$ vanishes, as it is a retract. 
Moreover, 
the spectrum $BP \wedge \Res L_{ev}$ is even, as follows from \cite[Theorem C]{Ravenel93}, hence $\underline{\pi}_{n\rho-1}(BP_\R\wedge L_{ev})=0$. 
Since $\underline{\pi}_{n\rho}(BP_\R\wedge L_{ev})$ is a retract of $\underline{\pi}_{n\rho}(BP_\R\wedge \Omega^\rho S^{\rho+1})$, and the restriction map is an isomorphism in the latter, it is also an isomorphism in the former.
\end{proof}

\begin{rmk}
    \cref{prop:resinjective} gives a calculation of $(BP_\R)_{*\rho}\Omega^\rho S^{\rho+1}$,
    relative to the calculation of $BP_*\Omega^2 S^3$ given in
    \cite[Theorem C]{Ravenel93}. 
    {However, we note that Ravenel only computes $BP_*\Omega^2S^3$ up to extensions in the Adams SS.} 
    \cref{prop:resinjective} and \cref{prop:evenoddsplit} together say that $BP_\R\wedge \OmRhoSRho$ splits as a sum of a strongly even $C_2$-spectrum and the suspension of a strongly even $C_2$-spectrum.
\end{rmk}

\section{Proof of \cref{lemma:integraldomain}}
\label{ProofLemma99}

Here, we give the delayed argument for the technical \cref{lemma:integraldomain}. Again, the argument is adapted from the proof of \cite[Lemma 4.15(b)]{RavenelWilson77}. We state it for convenience of the reader.

\LemmaNineNine*

\begin{proof}
    The ring $C_k$ is of the form $S_k[\asig]$, so 
    it suffices to prove the claim instead for the ring
        \[
            S_k:=\frac{\F_2[\usig^{\pm 2^k},v_{k},\ldots,t_0,e_1,e_2,\ldots]}{(t_0^{2^{k+1}}+\usig^{2^k}e_1^{2^{k}},\overline{r_{k+2}},\ldots)}.
        \]
For notational convenience, we will further replace $S_k$ with the ring
\begin{equation*}\label{eq:Rknring}
    R_{k}:=\F_2[w^\pm,v_0,v_1,\ldots,e_0,e_1,e_2,\ldots]/(r_{1,k},r_{2,k},r_{3,k},\ldots),
\end{equation*}
where $r_{1,k}=e_0^{2^{k+1}}+we_1^{2^k}$ and for $n>1$,
\[
r_{n,k}=\sum\limits_{i=0}^{n-1}v_ie_{n-i}^{2^{i+k}};
\]
we note that there is a ring isomorphism $S_k\cong R_{k}$ sending $v_i\mapsto v_{i-k}$, $t_0\mapsto e_0$, $\usig^{2^k}\mapsto w$, and $e_i\mapsto e_i$. For the purpose of making inductive degree arguments, we regard $R_k$ as a singly graded ring with $|w|=0$, $|v_j|=2(2^{j+k}-1)$ for all $j\ge0$, $|e_i|=2(2^i-1)$ for all $i>0$, and $|e_0|=1$.

To prove that $R_k$ is a domain, we follow \cite[Theorem~4.15(b)]{RavenelWilson77}. 
We write 
\[
    A=\F_2[w^\pm,v_0,v_1,\ldots,e_0,e_1,e_2,\ldots],
\]    
and we will show that the ideal $(r_{1,k},r_{2,k},\dots)\subset A$ is prime. We fix $n\ge 0$ 
and define the quotient ring 
\[
    A_i=A/(e_0,e_1,\ldots,e_{n-i}) \qquad \text{for $i\le n$}
\]
 and $A_{n+1}=A$. We let $J_i\subset A_i$ be the ideal $J_i=(r_{n,k},r_{n-1,k},\ldots,r_{n-i+1,k})$ for $i\le n$, and let $B_i=A_i[e_{n-i+1}^{-1}]$ for $i\le n+1$. 
 
 We first prove by induction that, for all $i\le n$, the ideal $J_i$ is regular in $A_i$. In the base case $i=1$, we have that
\[A_1=\F_2[w^\pm,v_0,v_1,\ldots,e_n,e_{n+1},\ldots]\]
is a domain, and $J_1=(r_{n,k})$ is a principal ideal generated by the nonzero element $r_{n,k}\equiv e_n^{2^k}v_0\in A_1$, so the claim follows. We fix an $i\le n$ and assume by induction that $J_{i-1}$ is regular in $A_{i-1}$. Via the short exact sequence
\[0\to A_i\xrightarrow{e_{n-i+1}}A_i\to A_{i-1}\to0\]
and the fact that $e_{n-i+1}$ has positive degree, a straightforward induction argument on degree shows that $J_{i-1}$ is also regular in $A_i$. If we can show that $J_{i-1}$ is prime in $A_i$, 
then the induction step will follow
once we observe that $r_{n-i+1,k}$ is nonzero in  $A_i/J_{i-1}$, since $A_i/J_{i-1}$ will be a domain. But we have that $A_i/J_{i-1}$ agrees with $A_i$ in the degree of $r_{n-i+1,k}$, and $r_{n-i+1,k}\equiv v_0e_{n-i+1}^{2^k}\in A_i$ when $i<n$ and $r_{1,k}\equiv we_1^{2^k}\in A_n$. So $r_{n-i+1,k}$ is nonzero in $A_i/J_{i-1}$.

To see that $J_{i-1}$ is prime in $A_i$, we note first that it is prime in $B_i$. Indeed, we have that $B_i\cong \F_2[w^\pm,v_0,v_1,\ldots,e_{n-i+1}^{\pm},e_{n-i+2},\ldots]$, and for each $r_{n-j,k}\in J_{i-1}$, we have that 
\[
    r_{n-j,k}=v_{i-j-1}e_{n-i+1}^{2^{i-j-1+k}}+\sum\limits_{l\in\{0,\ldots,n-j-1\}\setminus\{i-j-1\}}v_le_{n-j-l}^{2^{l+k}}.
\]
This implies that
there is a ring automorphism of $B_i$ replacing $r_{n-j,k}$ with $v_{i-j-1}$, from which it follows that $J_{i-1}$ is prime in $B_i$. 
Suppose then that $J_{i-1}$ is not prime in $A_i$, so that there exist $x,y\in A_i$ with $xy\in J_{i-1}$ and $x,y\notin J_{i-1}$. Since $J_{i-1}$ is prime in $B_i$, we can assume without loss of generality that $e_{n-i+1}^Nx\in J_{i-1}$ for some $N>0$ minimal. It follows that
\[e_{n-i+1}^Nx=\sum\limits_{j=1}^{i-1} a_jr_{n-j+1,k}\]
for some $a_j\in A_i$, and not all $a_j$ are divisible by $e_{n-i+1}$. Indeed, since $e_{n-i+1}$ is a nonzero element in the domain $A_i$, if all $a_j$ were divisible by $e_{n-i+1}$, we could use cancellation to see that $e_{n-i+1}^{N-1}x\in J_{i-1}$, contradicting minimality of $N$. Projecting to $A_{i-1}$ this gives
\begin{equation}\label{eq:primeidealexpression}
    0=\sum\limits_{j=1}^{i-1} a_jr_{n-j+1,k}\in A_{i-1}
\end{equation}
with not all $a_j=0$. 
We can further assume that if $j>1$ and $a_j\neq 0\in A_{i-1}$, then $a_j\notin J_{j-1}\subset A_{i-1}$. Indeed, if $a_j\in J_{j-1}$, then $a_j$ can be expressed as a linear combination of the $r_{n-l+1,k}$'s for $l<j$ and we can regroup the expression \eqref{eq:primeidealexpression} to have no $r_{n-j+1,k}$ term. The expression \eqref{eq:primeidealexpression} now contradicts regularity of $J_{i-1}$ in $A_{i-1}$.

Taking $i=n$, we have shown that $J_{n}$ is regular in $A_{n}$ and that $J_{n-1}$ is prime in $A_n$. In fact, there are ring isomorphisms $A_{n+1}\cong A_n[e_0]$ and $B_{n+1}\cong A_n[e_0^\pm]$, and since $e_0$ does not appear in any of the relations $r_{j,k}$ for $j>1$, we have that $A_{n+1}/J_{n-1}\cong (A_n/J_{n-1})[e_0]$ and $B_{n+1}/J_{n-1}\cong (A_n/J_{n-1})[e_0^\pm]$. It follows that $J_{n-1}$ is prime in both $A_{n+1}$ and $B_{n+1}$. We claim also that $J_n=J_{n-1}+(r_{1,k})$ is prime in $B_{n+1}$, which follows from the identifications
\begin{align*}
    B_{n+1}/J_n&=(B_{n+1}/J_{n-1})/(r_{1,k})\\
    &=(A_n/J_{n-1})[e_0^\pm]/(e_0^{2^k}+we_1^{2^k})\\
    &=(A_n/(w-1,J_{n-1}))[e_0^\pm,e_1^{-1}].
\end{align*}
Since $(A_n/(w-1,J_{n-1}))[e_0^\pm]$ is a domain, it suffices to observe that the localization $(A_n/(w-1,J_{n-1}))[e_0^\pm,e_1^{-1}]$ is nonzero, which follows from the fact that $J_{n-1}$ is prime in $B_n=A_n[e_1^\pm]$.

We have shown that $J_n$ is prime in $B_{n+1}$ and, to show that $J_n$ is also prime in $A_{n+1}=A$, we finish by mimicking the above argument showing that $J_{i-1}$ is prime in $A_i$. In fact we only used that $J_{i-1}$ was regular in $A_{i-1}$, that $J_{i-1}$ is prime in $B_i$, and that $A_i$ was a domain. The argument therefore goes through without change and we conclude that the ideal $I_n=(r_{n,k},\ldots ,r_{1,k})$ is prime in $A$ for all $n$. It follows that $R_k= A/I_\infty$ is a domain since in any fixed degree $R_k\cong A/I_n$ for $n$ sufficiently large.
\end{proof}

\section{Charts} \label{sec:charts}

We include here a few charts, to illustrate our computation described in \cref{cor:EinftyBSS} and \cref{subsec:hiddenextensions}.
For comparison, we also include on page \pageref{fig:BPOm2S3} a chart partially depicting the nonequivariant computation $BP_* \Omega^2 S^3$.

Additively, the computation of ${BP_\R}_\bigstar \Omega^\rho S^{\rho+1}$ splits into the computation of the $BP_\R$-homology of the Snaith summands, as in \cref{cor:HlgySnaithSplitting}. We have chosen to display (in a range) the $BP_\R$-homology of the Snaith summands of weights 0, 2, and 4. Recall that acccording to \cref{prop:evenoddsplit}, multiplication by $t_0$ gives an isomorphism from the even summands to the odd summands. In order to reduce visual clutter, we have further split the chart for the weight 4 homology into 2 charts, which we label as weight 4a and 4b.

Here is a key for reading the  charts.

\begin{enumerate}
    \item vertical lines denote multiplication by $v_0=2$
    \item horizontal lines denote multiplication by $\asig$
    \item slope 1/2 lines in \cref{fig:BPOm2S3} denote multiplication by $v_1$. In order to avoid clutter, we have not depicted the $v_1$-multiplications in weight 0. 
    \item each class is displayed using a symbol, indicating its periodicity with respect to a power of $\usig$ as labeled in the key. Thus the class $v_1$ in stem 2 and filtration 1 of the weight 0 chart contributes an $\F_2[\usig^{\pm 4}]$.
\end{enumerate}

We have also decided not to depict the $\asig$-multiples of $v_n$ for $n\geq 5$. In particular, the weight 0 chart displays only the $\asig$-multiples of $v_1$, $v_2$, $v_3$, and $v_4$ in Adams filtration 1.

\begin{rmk}
Recall from \cref{prop:E1identifications} that in the $\asig$-Bockstein spectral sequence, the class $t_0^2$ is equal to $\usig e_1$. The reader may replace the labels $t_0^2$ with $\usig e_1$ in the figures below, to emphasize the relation to $\usig$-periodicity, though this description is not valid in Ext groups. In particular, $\usig e_1$ is not a cycle in the cobar complex.
\end{rmk}

\clearpage

\begin{figure}[h]
\caption{The $E_2$-page for  summands of $BP\R_\bigstar\Omega^\rho S^{\rho+1}$}
\label{fig:Snaith0and2}

\includegraphics[width=0.95\textwidth]{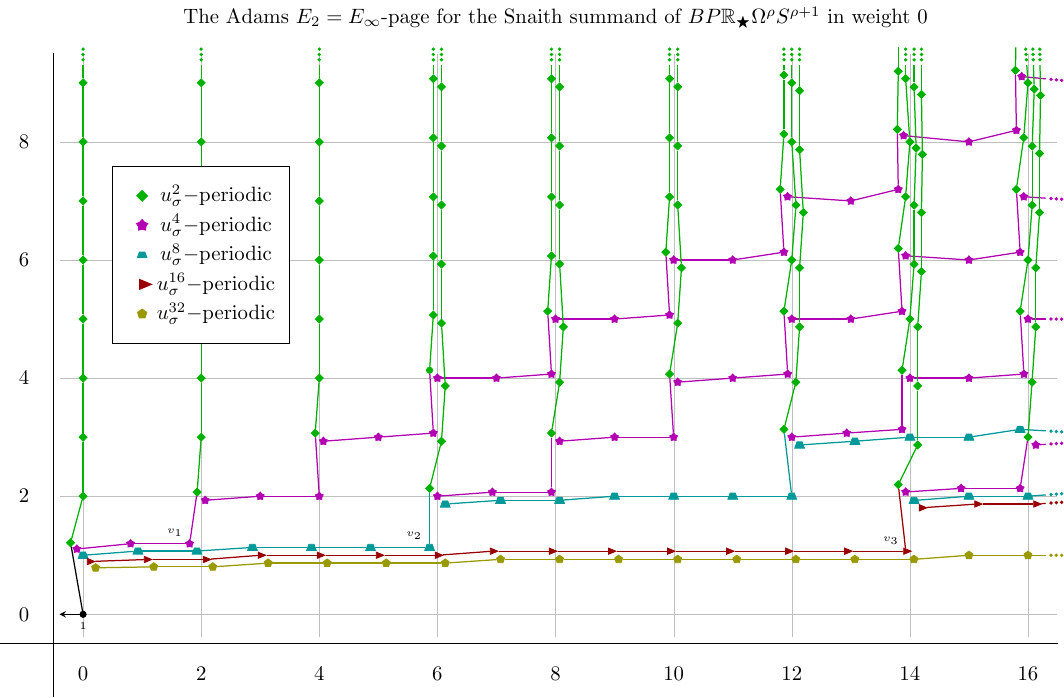}

\includegraphics[width=0.95\textwidth]{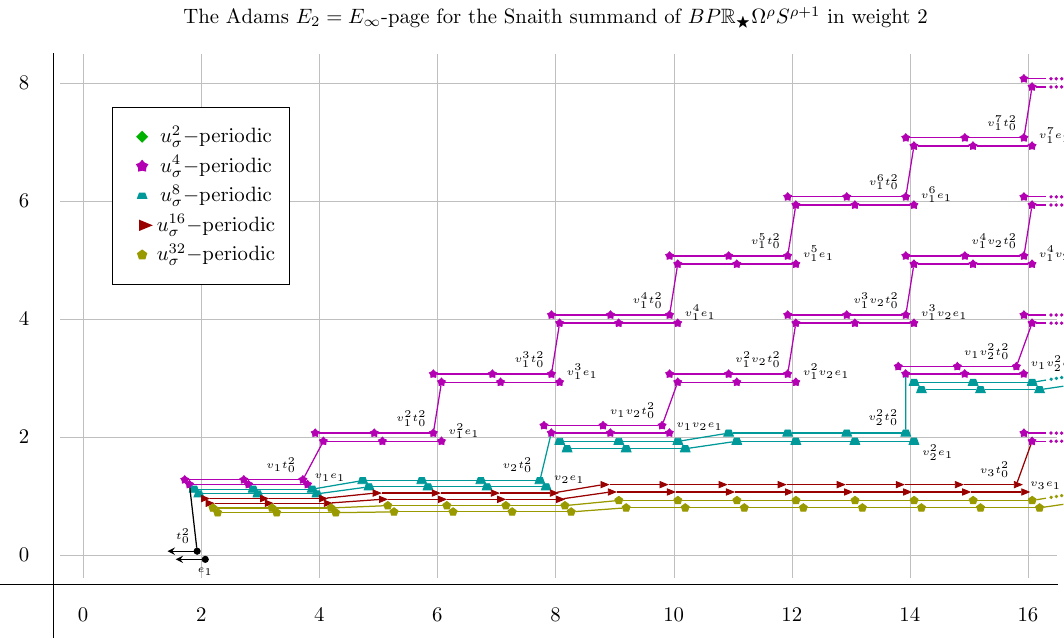}

\end{figure}

\clearpage

\label{fig:Snaith4}

\includegraphics[width=0.95\textwidth]{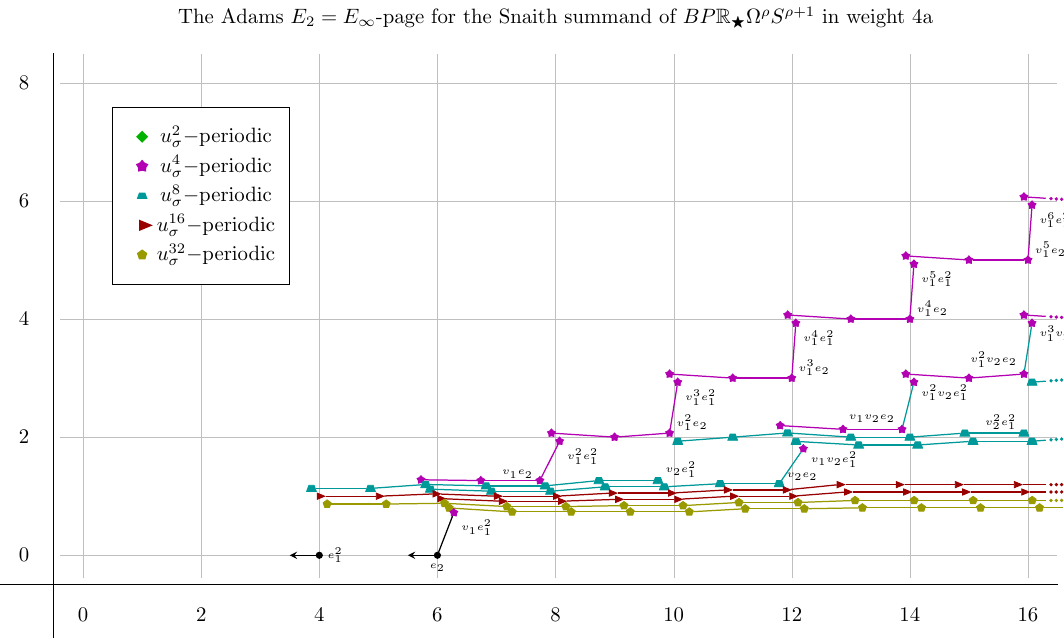}

\includegraphics[width=0.95\textwidth]{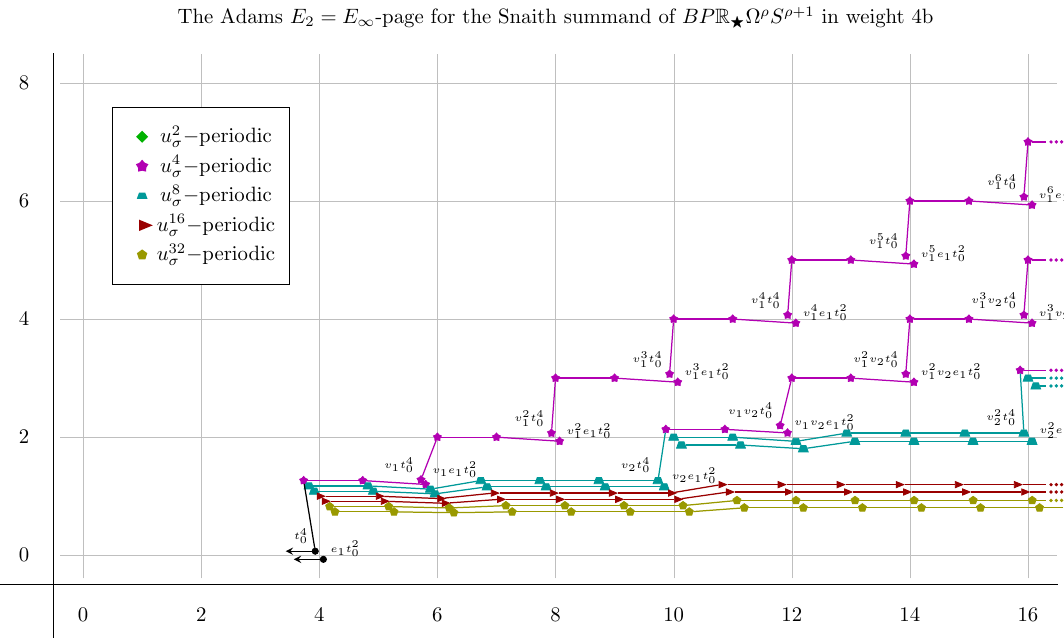}

\clearpage

\begin{figure}[h]
\caption{}
\label{fig:BPOm2S3}

\includegraphics[width=0.95\textwidth]{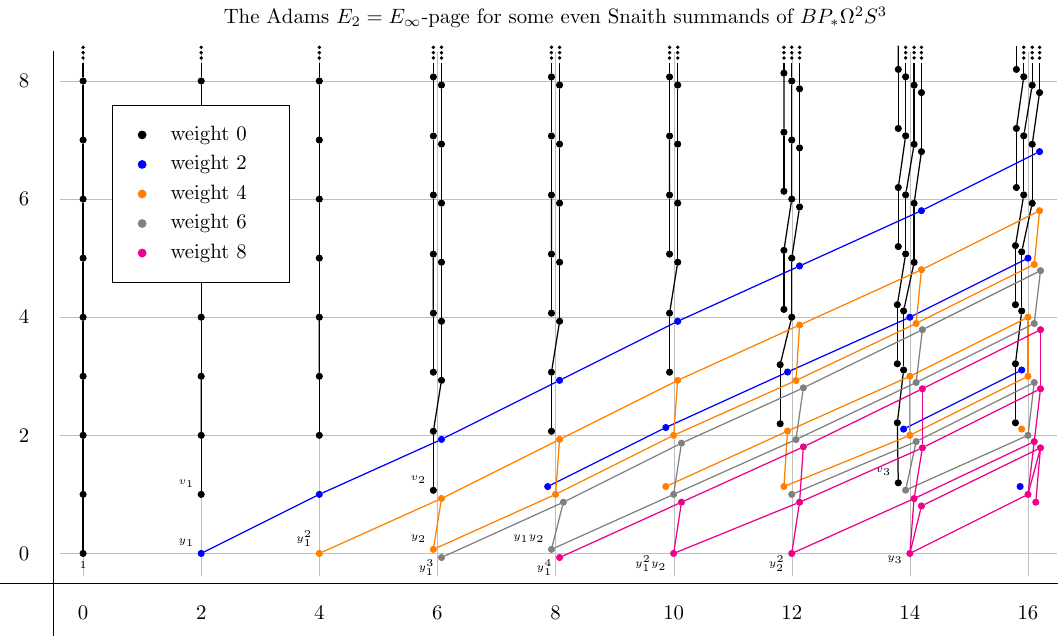}

\end{figure}

\bibliographystyle{alphaurl}
\bibliography{bib}

\end{document}